\documentclass[reqno,twoside,11pt]{amsart}
\usepackage{amsmath, amsfonts, amssymb, esint, amsthm, nicefrac, bbm, dsfont}
\usepackage{epsfig, graphicx}

\newcommand{\dx}{~\mathrm{d}x}
\newcommand{\sym}{\mathrm{sym}}
\newcommand{\curl}{\mathrm{curl}}
\newcommand{\so}{\mathrm{so}}
\newcommand{\Id}{\mathrm{Id}}

\newcommand{\R}{\mathbb{R}}
\def\endproof{\hspace*{\fill}\mbox{\ \rule{.1in}{.1in}}\medskip }
\newcommand*{\dbar}[1]{\bar{\bar{#1}}}

\numberwithin{equation}{section}
\theoremstyle{plain}

\newtheorem{theorem}{Theorem}[section]
\newtheorem{lemma}[theorem]{Lemma}
\newtheorem{corollary}[theorem]{Corollary}

\newtheorem{definition}[theorem]{Definition}

\theoremstyle{definition}

\begin{document}
\title[The Monge-Amp\`ere system]{The Monge-Amp\`ere system: convex integration in arbitrary
  dimension and codimension}
\author{Marta Lewicka}
\address{M.L.: University of Pittsburgh, Department of Mathematics, 
139 University Place, Pittsburgh, PA 15260}
\email{lewicka@pitt.edu} 

\thanks{M.L. was partially supported by NSF grant DMS-2006439. 
AMS classification: 35Q74, 53C42, 35J96, 53A35}

\begin{abstract}
In this paper, we study flexibility of weak solutions to the Monge-Amp\`ere system (MA) via convex
integration. 
This new system of Pdes is an extension of the Monge-Amp\`ere
equation in $d=2$ dimensions, naturally arising  from the prescribed curvature
problem and closely related to the classical problem of isometric immersions (II).

Our main result achieves density in the set of subsolutions, of the H\"older 
$\mathcal{C}^{1,\alpha}$ solutions to the Von K\'arm\'an system (VK)
which is the weak formulation of (MA). The regularity exponent $\alpha$ is any
exponent satisfying $\alpha<\frac{1}{1+
  d(d+1)/k}$ where $d$ is an arbitrary dimension and $k$ an
arbitrary codimension of the problem. At $k=1$, this agrees with the 
regularity $\mathcal{C}^{1,\alpha}$ for (II) with any $\alpha 
<\frac{1}{1+d(d+1)}$, proved by Conti, Delellis and Szekelyhidi in \cite{CS}. 
At $d=2, k=1$, this extends the initial findings by the author and Pakzad in
\cite{lewpak_MA} for (MA).

Our result seems to be optimal, from the technical viewpoint, for the
corrugation-based convex integration scheme. In particular, it covers the 
codimension interval $k\in \big(1, d(d+1)\big)$ so far uncharted 
even for the system (II), since the 
regularity $\mathcal{C}^{1,\alpha}$ with any $\alpha <1$ achieved by
K\"allen in \cite{Kallen}, strictly requires a large codimension.
Our second main result reproduces K\"allen's result in the context of (MA), obtaining
density in the set of
subsolutions, of $\mathcal{C}^{1,\alpha}$ regular solutions for any
$\alpha<1$ whenever $k\geq d(d+1)$. 

As an application of our results for (VK), we derive an energy scaling bound in the quantitative
immersability of Riemannian metrics, for nonlinear energy functionals 
modelled on the energies of deformations of thin prestrained
films in the nonlinear elasticity \cite{lew_book}.
\end{abstract}

\maketitle

\section{Introduction}

This paper concerns regularity and density of solutions to a new system of Pdes, called the
Monge-Amp\`ere system (\ref{MA}), which is the 
multi-dimensional version of the Monge-Amp\`ere equation arising from the prescribed curvature
problem.  As explained below, (\ref{MA}) is also closely related to the problem of isometric immersions
and the dimension reduction of thin films. Namely, given 
$F:\omega\to\R^{d^4}$ on a domain $\omega\subset \R^d$, we look for a
vector field $v$ such that:
\begin{equation*}\tag{MA}
\begin{split}
& v:\omega\to \R^k, 
\\ & \mathfrak{Det} \,\nabla^2v \doteq \big[\langle \partial_i\partial_s
v, \partial_j\partial_tv\rangle -
\langle \partial_i\partial_tv, \partial_j\partial_sv\rangle\big]_{i,j,
  s,t:1\ldots d}= F\quad\mbox{ in }\;\omega.
\end{split}
\end{equation*}
When $d=2$, $k=1$, recall that the Gaussian curvature of a surface described as the
graph of $v$ is $\kappa = \frac{\det\nabla^2 v}{(1+|\nabla
v|^2)^2}$. Replacing $v$ by a family of shallow displacements
$\{\epsilon v\}_{\epsilon\to 0}$, we obtain:
$$\kappa = \frac{\epsilon^2 \det\nabla^2 v}{(1+\epsilon^2|\nabla
  v|^2)^2} = \epsilon^2 \det\nabla^2v + o(\epsilon^2),$$
which yields the classical Monge-Amp\'ere equation:
\begin{equation*}
\begin{split}
& \, v:\omega\to \R, 
\\ & \det\nabla^2 v  \doteq \partial_{11}
v  \partial_{22} v - (\partial_{12} v)^2 = f\quad\mbox{ in }\;\omega,
\end{split}
\end{equation*}
as the prescription of the (leading order term of) Gaussian curvature of a
shallow surface $\{(x,\epsilon v(x)); \; x\in \omega\}\subset\R^{3}$.
To apply the same heuristics in arbitrary dimension
$d$ and codimension $k$, we consider the family of Riemannian metrics
generated by immersions $\{u^\epsilon=(id_d, \epsilon
v)\}_{\epsilon\to 0}$ and compute their Riemann curvatures in:
\begin{equation*}
\begin{split}
{Riem}\big((\nabla u)^T\nabla u\big) & =
{Riem}\big(\mbox{Id}_d + \epsilon^2(\nabla v)^T\nabla v\big) = 
-\frac{\epsilon^2}{2}\mathfrak{C}^2\big((\nabla v)^T\nabla v\big) + o(\epsilon^2)
\\ & = \epsilon^2 \mathfrak{Det} \,\nabla^2v + o(\epsilon^2).
\end{split}
\end{equation*}
The second order, linear operator $\mathfrak{C}^2$ is given in (\ref{C2}) below, and it reduces to 
$curl\,curl$ in our previous computation of $\kappa$, consistent with
having in that context: 
$\mbox{curl}\,\mbox{curl}(\nabla v\otimes \nabla v)= -2\det\nabla^2v$. 
Thus, the problem (\ref{MA}) arises by prescribing the (leading order
terms of) full Riemann
curvature tensor of the shallow manifold $\{(x,\epsilon v(x)); \; x\in \omega\}\subset\R^{d+k}$.

\bigskip

\noindent A necessary condition for (\ref{MA}) to be well posed is
that $F\in Range(\mathfrak{C}^2)$, with the equivalent compatibility conditions for
this to hold, specified in (\ref{sym_B}). Under these conditions
$F=-\mathfrak{C}^2(A)$ for some matrix field $A:\omega \to \R^{d\times
  d}_\sym$, and consequently (\ref{MA}) can be restated as:
\begin{equation*}
\begin{split}
& v:\omega\to \R^k, 
\\ & \mathfrak{C}^2(\frac{1}{2} (\nabla v)^T\nabla v) = \mathfrak{C}^2(A)\quad\mbox{ in }\;\omega.
\end{split}
\end{equation*}
Observing that $Kernel(\mathfrak{C}^2)$ consists of
symmetrized gradients, the above reduces to the weak formulation of
(\ref{MA}) called the Von K\'arm\'an system, in which we look for $v$,
$w$ such that:
\begin{equation*}\tag{VK}\label{VK}
\begin{split}
& v:\omega\to \R^k, \quad w:\omega\to\R^d,
\\ & \frac{1}{2}(\nabla v)^T\nabla v + \sym \nabla w = A\quad\mbox{ in }\;\omega.
\end{split}
\end{equation*}
When $d=2$, $k=1$, the left hand side of (\ref{VK}) is known in the theory of
elasticity as the the Von
K\'arm\'an stretching
content whose energy measures the stretching of a thin film with midplate
$\omega$, subject to the out of plane displacement $v$ and the in
plane displacement $w$. 

\bigskip

\noindent The closely related problem to (\ref{MA}) and (\ref{VK}) is the problem
of finding an isometric immersion $u$ of the given Riemannian metric
$g:\omega\to\R^{d\times d}_{\sym, >}$, into a higher dimensional space $\R^{d+k}$:
\begin{equation*}\tag{II}\label{II}
\begin{split}
& u:\omega\to \R^{d+k},
\\ & (\nabla u)^T\nabla u = g\quad\mbox{ in }\;\omega.
\end{split}
\end{equation*}
Indeed, (\ref{II}) reduces to (\ref{VK}) when equating the leading order terms in
the family of Riemannian metrics $\{\Id_d+2\epsilon^2 A\}_{\epsilon\to 0}$
and the metrics generated by the immersions $\{\bar u^\epsilon =
({id}_d+\epsilon^2w, \epsilon v)\}_{\epsilon\to 0}$:
$$(\nabla \bar u^\epsilon)^T\nabla \bar u^\epsilon = \mbox{Id}_d +
\epsilon^2 \big((\nabla v)^T\nabla v + 2\, \sym\nabla w\big) + o(\epsilon^2).$$
In this sense, the three problems (\ref{MA}), (\ref{VK}) and
(\ref{II}) are intrinsically related. 

\bigskip

\noindent The purpose of this paper is to
investigate existence of H\"older continuous solutions to (\ref{VK})
and (\ref{MA}), using the convex integration technique and motivated
by the existing applications of this technique to (\ref{II}).
We recall that when posed in arbitrary dimension $d$ but codimension $k=1$,
it has been shown in \cite[Theorem 1.1]{CDS} that any local
subsolution to (\ref{II}) can be uniformly approximated
by a sequence of solutions $\{u_n\}_{n=1}^\infty$ of regularity
$\mathcal{C}^{1,\alpha}$, for any H\"older exponent
$\alpha<\frac{1}{1+2d_*}$ where  $d_*=d(d+1)/2$ is the
dimension of $\R^{d\times d}_\sym$. On the other hand, as showed in
\cite{Kallen}, regularity $\mathcal{C}^{1,\alpha}$ with any $\alpha <1$ can be achieved in
sufficiently high codimension $k$; however even for the local result
this argument strictly requires $k\geq 2d_*$, whereas it yields no outcome for
$k<2d_*$.  These two results, albeit both relying  
on convex integration, use different
constructions of the cascade of perturbations: Kuiper's corrugations
in \cite{CDS} and Nash's spirals in \cite{Kallen}, and one cannot be
deduced from the other. 

\medskip

\noindent To our knowledge, there has been no result
for the codimension interval $k\in (1, 2d_*)$
interpolating the regularity in \cite{CDS} and \cite{Kallen}, 
or even improving the exponent $\frac{1}{1+2d_*}$ without the requirement $k\geq 2d_*$.
\footnote{ $\,$After submission of this paper, there appeared a new preprint \cite{CS_new} which
states flexibility up to $\mathcal{C}^{1,\frac{1}{d+2}}$ 
for (\ref{II}) when $k=d$. This is consistent with our
result, as $\frac{1}{1+2d_*/d} = \frac{1}{d+2}$. If fact, we
expect that the same techniques as in the present paper may be
applied for the system (\ref{II}) as well, with the same regularity
exponents as in Theorem \ref{th_final} holding locally and for compact
manifold cases.}
In our paper we achieve precisely this goal, for the
system (\ref{VK}). Our main result states that any $\mathcal{C}^1$-regular pair $(v,w)$ 
which is a subsolution of (\ref{VK}), can be uniformly approximated
by a sequence of solutions $\{(v_n,w_n)\}_{n=1}^\infty$ of regularity
$\mathcal{C}^{1,\alpha}$ for any H\"older exponent
$\alpha<\frac{1}{1+2d_*/k}$, in case of {\em arbitrary $d$ and $k$}. Our proof
only uses corrugations, extending the construction in \cite{CDS}
in an optimal manner. Clearly, the obtained critical regularity exponent $1/2$
at $k=2d_*$ is inferior to the exponent $1$ from a version of the same
construction as in \cite{Kallen}, that we also demonstrate in our
paper. We expect that the superposition of both techniques should
yield a tighter interpolation, which is the subject of the ongoing research.

\bigskip

\noindent We state our results and offer further discussion on the
relation among the systems (\ref{MA}), (\ref{VK}), (\ref{II}), as well
as their application to the energy scaling bound for thin
multidimensional films, in the subsections below.

\subsection{Convex integration by corrugations, arbitrary $\mathbf{d}$ and $\mathbf{k}$}
The following theorem is our main result.  We refer to it as {\em flexibility} of (\ref{VK}) up to
$\mathcal{C}^{1,\frac{1}{1+2d_*/k}}$:

\begin{theorem}\label{th_final}
Let $\omega\subset\R^d$ be an open, bounded domain. Given
two vector fields $v\in\mathcal{C}^1(\bar\omega,\R^k)$, $w\in\mathcal{C}^1(\bar\omega,\R^d)$ and
a matrix field $A\in\mathcal{C}^{0,\beta}(\bar\omega,\R^{d\times d}_\sym)$, assume that:
$$\mathcal{D}=A-\big(\frac{1}{2}(\nabla v)^T\nabla v + \sym\nabla
w\big) \quad\mbox{ satisfies } \quad \mathcal{D}>c\,\Id_d \; \mbox{ on } \; \bar\omega,$$
for some $c>0$, in the sense of matrix inequalities. Fix $\epsilon>0$ and let:
$$0<\alpha<\min\Big\{\frac{\beta}{2},\frac{1}{1+d(d+1)/k}\Big\}.$$ 
Then, there exists $\tilde v\in\mathcal{C}^{1,\alpha}(\bar \omega,\R^k)$ and
$\tilde w\in\mathcal{C}^{1,\alpha}(\bar\omega,\R^d)$ such that the following holds:
\begin{align*}
& \|\tilde v - v\|_0\leq \epsilon, \quad \|\tilde w - w\|_0\leq \epsilon,
\tag*{(\theequation)$_1$}\refstepcounter{equation} \label{FFbound1}\vspace{1mm}\\ 
& A -\big(\frac{1}{2}(\nabla \tilde v)^T\nabla \tilde v + \sym\nabla
\tilde w\big) =0 \quad \mbox{ in }\;\bar\omega.
\tag*{(\theequation)$_2$} \label{FFbound2}
\end{align*}
\end{theorem}

\smallskip

\noindent This result generalizes \cite[Theorem 1.1]{lewpak_MA},
where we proved flexibility for (\ref{VK}) up to $\mathcal{C}^{1,\frac{1}{7}}$ in dimensions
$d=2$, $k=1$. In that special case, motivated
by theory of elasticity, the left hand side of (\ref{VK}) represents the Von K\'arm\'an
content $\frac{1}{2}\nabla v\otimes \nabla v +\sym\nabla w$
written in terms of the scalar out of plane displacement $v$ and the
in plane displacement $w$ of the middle plate $\omega$ of a thin
film. The case $d=2$ is special and flexibility (in codimension $1$) of (\ref{VK}) holds
up to $\mathcal{C}^{1,\frac{1}{5}}$ as shown in \cite[Theorem 1.1]{CS}
using the conformal equivalence of $2$-dimensional metrics to the Euclidean metric.
In our first extension \cite{lew_improved} of the present work,
we likewise show that any $k\geq 1$ allows for flexibility up to
$\mathcal{C}^{1,\frac{1}{1+4/k}}$  when $d=2$. 
\footnote{ $\;$When $d=k=2$, the same result for (\ref{II}) has been included in
\cite{CS_new} as flexibility up to $\mathcal{C}^{1,\frac{1}{3}}$. }
After the submission of
our both works, we learned of the recent preprint \cite{CHI} in which
flexibility of (\ref{VK}) for $d=2$, $k=1$ has been further improved to
hold up to $\mathcal{C}^{1,\frac{1}{3}}$.
In our second extension \cite{lew_improved2} we combined the technique
of \cite{CHI} to show flexibility up to
$\mathcal{C}^{1,\frac{2^k-1}{2^{k+1}-1}}$ for any $k$ and up to
$\mathcal{C}^{1,1}$ for $k\geq 4$ when $d=2$.

\bigskip

\noindent The main new technical ingredient allowing for the flexibility range
stated in Theorem \ref{th_final}, is the following ``stage''-type
construction in the convex integration algorithm for (\ref{VK}):

\begin{theorem}\label{thm_stage}
Let the vector fields $v\in\mathcal{C}^2(\bar \omega,\R^k)$, $w\in\mathcal{C}^2(\bar\omega,\R^d)$ and
the matrix field $A\in\mathcal{C}^{0,\beta}(\bar\omega,\R^{d\times d}_\sym)$ be given on an open,
bounded domain $\omega\subset\R^d$. Assume that:
$$\mathcal{D}=A -\big(\frac{1}{2}(\nabla v)^T\nabla v + \sym\nabla
w\big) \quad \mbox{ satisfies } \; 0<\|\mathcal{D}\|_0\leq 1.$$
Fix two constants $M, \sigma$ such that:
$$M\geq\max\{\|v\|_2, \|w\|_2, 1\} \quad\mbox{ and }\quad \sigma\geq 1.$$
Then, there exist $\tilde v\in\mathcal{C}^2(\bar \omega,\R^k)$ and
$\tilde w\in\mathcal{C}^2(\bar\omega,\R^d)$ such that, denoting:
$$\tilde{\mathcal{D}}=A -\big(\frac{1}{2}(\nabla \tilde v)^T\nabla \tilde v + \sym\nabla \tilde w\big), $$
the following holds: 
\begin{align*}
& \|\tilde v - v\|_1\leq C\|\mathcal{D}\|_0^{1/2}, \quad \|\tilde w -
w\|_1\leq C\|\mathcal{D}\|_0^{1/2}(1+\|\nabla v\|_0),
\tag*{(\theequation)$_1$}\refstepcounter{equation} \label{Abound1}\vspace{1mm}\\
& \|\nabla^2\tilde v\|_0\leq CM\sigma^{d_*/k},\quad \|\nabla^2\tilde w\|_0\leq CM\sigma^{d_*/k}(1+\|\nabla v\|_0),
\tag*{(\theequation)$_2$}\label{Abound2} \\ 
& \|\tilde{\mathcal{D}}\|_0\leq C\Big(\frac{\|A\|_{0,\beta}}{M^\beta} \|\mathcal{D}\|_0^{\beta/2} +
\frac{\|\mathcal{D}\|_0}{\sigma}\Big), \tag*{(\theequation)$_3$} \label{Abound3}
\end{align*}
where $d_*=d(d+1)/2$ and where the constants $C$ depend only on $d, k$ and $\omega$.
\end{theorem}

\smallskip

\noindent We briefly outline how our construction differs from \cite{lewpak_MA}
and \cite{CDS}. There, a stage consisted of precisely $d_*$ ``steps'',
each cancelling one of the rank-one ``primitive'' deficits in the
decomposition of $\mathcal{D}$. The initially chosen frequency of
perturbation was multiplied by a factor $\sigma$ at each step, leading
to the increase of the second derivative by $\sigma^{d_*}$ and thus
to the exponent $d_*$ replacing $d_*/k$ in \ref{Abound2}, while the remaining error in
$\mathcal{D}$ was of order $1/\sigma$, leading to \ref{Abound3}. 

\medskip

\noindent Presently, we first observe that $k$ such
deficits may be cancelled at once, by using $k$ linearly independent
codimensions. Further, when all the first order primitive deficits are
cancelled, one may proceed to cancelling the second order deficits
obtained as the one-dimensional decompositions of the error between the
original and the decreased $\mathcal{D}$; the corresponding
frequencies must be then increased by the factor $\sigma^{1/2}$,
precisely due to the decrease of $\mathcal{D}$ by the factor
$1/\sigma$. One may inductively proceed in this fashion,
cancelling even higher order deficits, and adding $k$-tuples of single
codimension perturbations, for a total of
$N=lcm(k, d_*)$ steps. The frequencies get increased by the factor of $\sigma$ over each multiple
of $k$, leading to the increase of the second derivatives by $\sigma$, 
and by the factor of $\sigma^{1/2}$ over each multiple of $d_*$,
where the deficit decreases by the factor of $1/\sigma$. In the final count,
the total increase of the second derivatives has the factor
$\sigma^{N/k}$, while the decrease of the deficit has the factor $1/\sigma^{N/d_*}$. The relative
change of order is thus $(N/k) / (N/d_*) = d_*/k$, as stated in Theorem \ref{thm_stage}.

\medskip

\noindent We point out that for this scheme to work, it is essential to
use the optimal ``step''-type construction in which the chosen
one-dimensional primitive deficit is cancelled at the expense of
introducing least error possible. Our previous definition from
\cite{lewpak_MA} would not work for this purpose, and we need to
superpose three corrugations rather than two.

\subsection{Convex integration by spirals, $\mathbf{k\geq 2d_*}$}
For large codimensions, one can reach flexi\-bility of (\ref{VK}) up to $\mathcal{C}^{1,1}$,
motivated by a similar result for (\ref{II}) in \cite{Kallen}:

\begin{theorem}\label{th_finalK}
In the context of Theorem \ref{th_final}, assume that the codimension
$k$ satisfies $k\geq 2d_*=d(d+1)$. Then, the same result is valid for any exponent in the range:
$$ 0<\alpha<\min\Big\{\frac{\beta}{2},1\Big\}.$$
\end{theorem}

\noindent The ``stage'' construction allowing for flexibility as
above, is the counterpart of Theorem \ref{thm_stageK}:

\begin{theorem}\label{thm_stageK}
Let $\omega\subset\R^d$ and $k$ be as in Theorem \ref{th_finalK}. Fix an exponent $\delta>0$. Then,
there exists $\sigma_0>1$ depending only on $\omega$ and $\delta$,
such that we have the following. Given $v\in\mathcal{C}^2(\bar \omega,\R^k)$,
$w\in\mathcal{C}^2(\bar\omega,\R^d)$, 
$A\in\mathcal{C}^{0,\beta}(\bar\omega,\R^{d\times d}_\sym)$ and given
two constants $M, \sigma$ with the properties:
$$\mathcal{D}=A -\big(\frac{1}{2}(\nabla v)^T\nabla v + \sym\nabla
w\big) \quad \mbox{ satisfies } \; 0<\|\mathcal{D}\|_0\leq 1,$$
$$M\geq\max\{\|v\|_2, \|w\|_2, 1\}, \qquad \sigma\geq \sigma_0,$$
there exist $\tilde v\in\mathcal{C}^2(\bar \omega,\R^k)$,
$\tilde w\in\mathcal{C}^2(\bar\omega,\R^d)$ such that, denoting:
$$\tilde{\mathcal{D}}=A -\big(\frac{1}{2}(\nabla \tilde v)^T\nabla \tilde v + \sym\nabla \tilde w\big), $$
the following bounds are valid,  with constants $C$ depending only on
$d,k, \omega$ and $\delta$:
\begin{align*}
& \|\tilde v - v\|_1\leq C\|\mathcal{D}\|_0^{1/2}, \quad \|\tilde w -
w\|_1\leq C\|\mathcal{D}\|_0^{1/2}(1+\|\nabla v\|_0), \vspace{2mm}
\tag*{(\theequation)$_1$}\refstepcounter{equation} \label{Abound1K}\\
& \|\nabla^2\tilde v\|_0\leq CM\sigma^{\delta},\quad \|\nabla^2\tilde
w\|_0\leq CM\sigma^{\delta}(1+\|\nabla v\|_0), \vspace{2mm}
\tag*{(\theequation)$_2$}\label{Abound2K} \\ 
& \|\tilde{\mathcal{D}}\|_0\leq C\Big(\frac{\|A\|_{0,\beta}}{M^\beta} \|\mathcal{D}\|_0^{\beta/2} +
\frac{\|\mathcal{D}\|_0}{\sigma}\Big). \tag*{(\theequation)$_3$} \label{Abound3K}
\end{align*}
\end{theorem}

\medskip

\noindent An outline of this construction, based on the approach in \cite{Kallen}, is as follows. Firstly,
each rank-one ``primitive'' deficit is cancelled using two
codimensions, via spiral-like perturbations of the fields $v,w$,
rather than via one-dimensional corrugations. This allows for
a better order in the second order deficit. Since we
now have $2d_*$ codimensions available, we may rank-one decompose
the new deficit as well and cancel it right away by adjusting the original
perturbations. Proceeding this way, it is possible to cancel arbitrarily
high order of deficits, keeping the frequency at a chosen value $\sigma$ while
assuring that $\mathcal{D}$ is decreased by the
factor $1/\sigma^N$, for arbitrarily large $N$.

\subsection{The Monge-Amp\'ere system}

We now proceed to interpreting Theorem \ref{th_final} in the context of the
Monge-Amp\'ere system. 
Recall that for a matrix field $A=[A_{ij}]_{i,j=1\ldots 2}:\R^2\to\R^{2\times 2}$, the scalar
field ${curl}\, curl A$ is defined by taking 
the $curl$ operator on each row of $A$, and then applying another $curl$ on thus
formed two-dimensional vector field:
\begin{equation*}
\begin{split}
\curl\,\curl A & = \curl\big[\partial_1 A_{12}-\partial_2A_{11}, \partial_1
A_{22}-\partial_2A_{21}\big] \\ & = \partial_1 \partial_1A_{22}-\partial_1\partial_2A_{21}
- \partial_1\partial_2 A_{12} +\partial_2\partial_2A_{11}.
\end{split}
\end{equation*}
It is well known that the kernel of ${curl}\,
curl $ when restricted to $\R^{2\times 2}_{\sym}$ matrix fields, consists precisely
of symmetric gradients. We will be concerned with the following generalization of ${curl}\,
curl$, serving the same characterisation in higher dimensions:

\begin{definition}\label{defC2}
Given a $d$-dimensional square matrix field $A = [A_{ij}]_{i,j=1\ldots
  d}:\omega\to \R^{d\times d}$ on a domain $\omega\subset\R^d$, we
define $\mathfrak{C}^2(A):\omega\to \R^{d^4}$ by:
\begin{equation}\label{C2}
\mathfrak{C}^2(A)_{ij,st} = \partial_i\partial_s A_{jt}
+ \partial_j\partial_tA_{is} - \partial_{i}\partial_tA_{js}
- \partial_j\partial_sA_{it}\quad\mbox{ for all }\; i,j,s,t=1\ldots d.
\end{equation}
\end{definition}

\smallskip

\noindent It can be checked that the components of the Riemann curvature tensor
of a family of metrics $\Id_d+\epsilon A$ on $\omega$, are given, to the leading
order, by the components of $\mathfrak{C}^2(A)$:
\begin{equation}\label{RiemC2}
Riem(\Id_d + \epsilon A)_{ij,st} =
-\frac{\epsilon}{2}\mathfrak{C}^2(A)_{ij,st} + O(\epsilon^2) \quad\mbox{ for all }\; i,j,s,t=1\ldots d.
\end{equation}
For dimension $d=2$, the above formula yields the linearization of the
Gaussian curvature: $\kappa(\Id_2 + \epsilon A) =
-\frac{\epsilon}{2}\curl\,\curl A +O(\epsilon^2)$.
We have the following:

\begin{lemma}\label{th_equiv}
Let $\omega\subset \R^d$ be an open, bounded, contractible domain with
Lipschitz boundary. Given a symmetric matrix field $A\in
L^2(\omega,\R^{d\times d}_{\sym})$, the following conditions are equivalent:
\begin{itemize}
\item[(i)] $A=\sym\nabla w$ for some $w\in H^{1}(\omega,\R^d)$,
\item[(ii)] $\mathfrak{C}^2(A)=0$ in the sense of distributions on $\omega$.
\end{itemize}
\end{lemma}

\smallskip

\noindent For $A=(\nabla v)^T\nabla v$ given through a
vector field $v:\omega\to\R^k$, a direct calculation yields:
\begin{equation*}
\begin{split}
\mathfrak{C}^2\big((\nabla v)^T\nabla v\big)_{ij, st}
= 2 \langle \partial_i\partial_tv, \partial_j\partial_sv\rangle - 2 \langle \partial_i\partial_sv, \partial_j\partial_tv\rangle.
\end{split}
\end{equation*}
When $d=2$ and $k=1$, the above reduces to the familiar formula:
$\curl\,\curl (\nabla v\otimes\nabla v) = -2\det\nabla^2v$. Following
this motivation, we introduce:

\smallskip

\begin{definition}\label{defMA}
For $v: \omega\to \R^{k}$ defined on a domain $\omega\subset\R^d$, we
set $\mathfrak{Det}\,\nabla^2 v:\omega\to \R^{d^4}$ in:
\begin{equation}\label{DetMA}
\big(\mathfrak{Det}\,\nabla^2v\big)_{ij,st} =
\langle \partial_i\partial_sv, \partial_j\partial_tv\rangle -
\langle \partial_i\partial_tv, \partial_j\partial_sv\rangle 
\quad\mbox{ for all }\; i,j,s,t=1\ldots d.
\end{equation}
Given $F: \omega\to \R^{d^4}$, we call the following system of Pdes, the
Monge-Amp\'ere system:
$$\mathfrak{Det}\,\nabla^2v = F \quad\mbox{ on }\; \omega.$$
\end{definition}

\medskip

\noindent Lemma \ref{th_equiv} can be restated in this context as
follows. Given a matrix field $A:\omega\to\R^{d\times
  d}_\sym$ on a domain $\omega\subset\R^d$, the problem (\ref{VK}) is
equivalent to (disregarding the regularity questions):
\begin{equation*}\tag{MA}\label{MA}
\begin{split}
& v:\omega\to \R^k, \\ & \mathfrak{Det}\,\nabla^2v = - \mathfrak{C}^2(A),
\end{split}
\end{equation*}
which, for $d=2$ and $k=1$, is precisely the Monge-Amp\'ere constraint
$\det\nabla^2v = - \curl\,\curl A$ appearing in the 
dimensionally reduced, linearized Kirchhoff's theory of thin plates \cite{FJM}.
For the family of immersions:
$\bar u^\epsilon = id_d + \epsilon [0,v] +\epsilon^2[w,0]:\omega\to\R^{d+k}$,
one further notes that:
$$(\nabla\bar u^\epsilon)^T\nabla \bar u^\epsilon = \Id_d +
2\epsilon^2\big(\frac{1}{2}(\nabla v)^T\nabla v + \sym\nabla w\big) + O(\epsilon^4).$$
From (\ref{RiemC2}), we thus see that the problem of finding a vector field $v$ for
which the Riemann curvatures of the metrics $\Id_d+\epsilon^2 A$ and the Riemann
curvatures of the pull-back of $\Id_{d+k}$ via the reduced maps
$u^\epsilon$ below, coincide at their lowest order terms in $\epsilon$ on $\omega$:
\begin{equation}\label{pb4}
\begin{split}
& u^\epsilon = id_d + \epsilon [0,v]:\omega\to\R^{d+k},
\\ & Riem\big(\Id_d+\epsilon^2A\big) = Riem\big((\nabla
u^\epsilon)^T\nabla u^\epsilon\big) + o(\epsilon^2)
\end{split}
\end{equation}
is equivalent to the problem of finding $v$ that can be matched by an
auxiliary vector field $w$ so that the two Riemannian metrics families: $\Id_d+\epsilon^2 A$,
and the pull-back of $\Id_{d+k}$ via the maps
$\bar u^\epsilon$, coincide at their lowest order terms in $\epsilon$ on $\omega$:
\begin{equation}\label{pb5}
\begin{split}
& \bar u^\epsilon = id_d + \epsilon [0,v] + \epsilon^2[w,0]:\omega\to\R^{d+k},
\\ & \Id_d+\epsilon^2A = (\nabla \bar u^\epsilon )^T\nabla \bar u^\epsilon  + o(\epsilon^2).
\end{split}
\end{equation}
Thus, the four problems (\ref{VK}), (\ref{MA}), (\ref{pb4}) and (\ref{pb5}) are equivalent.

\bigskip

\noindent We further identify the range of $\mathfrak{C}^2$, in terms of the derived symmetry and
Bianchi identities:

\begin{lemma}\label{th_solve}
Let $\omega\subset\R^d$ be an open, bounded, contractible domain with
Lipschitz boundary. Given $F=[F_{ij, st}]_{i,j,s,t=1\ldots
  d}\in L^2(\omega, \R^{d^4})$, the following are equivalent:
\begin{itemize}
\item[(i)] $F=\mathfrak{C}^2(A)$ for some $A\in H^{2}(\omega,\R^{d\times d}_\sym)$,
\item[(ii)] $F$ satisfies the compatibility conditions, for all $i,j,s,t,q=1\ldots d$:
\begin{equation}\label{sym_B}
\begin{split}
& F_{ij,st} = - F_{ji,st} = -F_{ij, ts}, \qquad F_{ij, st} = F_{st,ij},
\\ & F_{ij, st}+F_{is,tj} + F_{it,js} = 0, 
\\ & \partial_qF_{ij,st} + \partial_sF_{ij,tq} + \partial_tF_{ij,qs} = 0
\quad \mbox{ in the sense of distributions on }\; \omega. 
\end{split}
\end{equation}
\end{itemize}
\end{lemma}

\smallskip

\noindent The above discussion motivates then the following:

\begin{definition}\label{def_weakMA}
Assume that $F\in L^2(\omega,\R^{d^4})$ given on an open, bounded, contractible domain
$\omega\subset\R^d$ with Lipschitz boundary, satisfies conditions (\ref{sym_B}). We say that
$v\in H^{1}_{loc}(\omega,\R^k)$ is a weak solution to the Monge-Amp\'ere system:
\begin{equation}\label{MAF}
\mathfrak{Det}\,\nabla^2v = F \quad\mbox{ on } \;\omega,
\end{equation}
provided that there exists $w\in W^{1,1}_{loc}(\omega,\R^d)$ such that
(\ref{VK}) holds with $\mathfrak{C}^2(A)= -F$, namely: 
$$\frac{1}{2}(\nabla v)^T\nabla v +\sym\nabla w =
-\big(\mathfrak{C}^2\big)^{-1}(F) \quad\mbox{ on } \;\omega.$$
\end{definition}

\noindent For $d=2$, $k=1$, any $F\in L^{1+}(\omega, \R)$ can be expressed as
the right hand side of (\ref{MA}), because writing $A=\gamma\Id_2$
where $\Delta\gamma = -F$ in $\omega$, there holds: $F=-\curl\,\curl
\,A$. In higher dimensions, the solvability conditions are nontrivial
and precisely given by (\ref{sym_B}) in
Theorem \ref{th_solve}. In view of Theorems \ref{th_final} and \ref{th_finalK}, we thus
obtain the following extension of \cite[Theorem 1.1]{lewpak_MA} 
proved there in dimension $d=2$ and codimension $k=1$, now to arbitrary $d,k$:

\begin{theorem}\label{th_CI_weakMA} 
Let $F\in L^{\infty} (\omega, \R^{d^4})$ on an open, bounded, contractible
domain $\omega\subset\mathbb{R}^d$ with Lipschitz boundary, satisfy
(\ref{sym_B}).  Fix $k\geq 1$ and fix an exponent $\alpha$ in:
$$0< \alpha< \frac{1}{1+d(d+1)/k}, \quad \mbox{ or } \; 0< \alpha< 1
\; \mbox{ in case of } \; k\geq d(d+1).$$ 
Then the set of $\mathcal{\mathcal{C}}^{1,\alpha}(\bar\omega, \R^k)$
weak solutions to (\ref{MAF}) is dense in $\mathcal{\mathcal{C}}^0(\bar\omega, \R^k)$. 
Namely, every $v\in \mathcal{\mathcal{C}}^0(\bar\omega,\R^k)$ is the
uniform limit of some sequence $\{v_n\in\mathcal{\mathcal{C}}^{1,\alpha}(\bar\omega,\R^k)\}_{n=1}^\infty$,
such that:
\begin{equation*} 
\mathfrak{Det}\, \nabla^2 v_n  = F \quad \mbox{ on } \; \omega,
\quad\mbox{ for all }\; n=1\ldots\infty.
\end{equation*}
\end{theorem} 

\subsection{Energy scaling bound for thin multidimensional films}

As an application, we now present an estimate on the energy
functional that is the generalisation to arbitrary dimension and
codimension, of the non-Euclidean
elasticity. For $d=2$, $k=1$, this functional models the
elastic energy of deformations of prestrained films, and various
techniques have been applied to its study \cite{lew_book}. From
another point of view, given the Riemannian metric $g$ on a reference
configuration $\Omega$, the energy 
$\mathcal{E}$ below measures the averaged pointwise deficit of an
immersion from being an orientation preserving isometric immersion of
$g$, for all weakly regular immersions.

\medskip

\noindent More precisely, given $\omega\subset\R^d$ we define the
family of ``thin films'', parametrised by $h\ll 1$:
$$\Omega^h = \big\{(x,z); ~x\in\omega, ~ z\in B(0,h)\subset\R^k\big\}\subset\R^{d+k}.$$
Consider the Riemannian metrics on $\Omega^h$ of the form:
$$g^h=\Id_{d+k}+2h^{\gamma/2}S, \quad\mbox{ where } \gamma>0 \mbox{
  and } S\in\mathcal{C}^\infty(\bar\omega,\R^{(d+k)\times (d+k)}_\sym).$$
We then pose the problem of minimizing the following energy functionals, as  $h\to 0$:
\begin{equation}\label{Eh}
\mathcal{E}^h(u) = \fint_{\Omega^h} W\big((\nabla u)(g^h)^{-1/2}\big)~\mathrm{d}(x,z)
\qquad\mbox{ for all }\; u\in H^{1}(\Omega^h,\R^{d+k}).
\end{equation}
The function $W:\R^{(d+k)\times (d+k)}\to [0,\infty]$ is assumed
to be $\mathcal{C}^2$-regular in the vicinity of $\mathrm{SO}(d+k)$, equal to
$0$ at $\Id_{d+k}$, and frame-invariant in the sense that $W(RF)= W(F)$
for all $R\in \mathrm{SO}(d+k)$. Questions on asymptotics
of minimizing configurations to $\mathcal{E}^h$ as
$h\to 0$, in function of the scaling exponent $\beta$ in: $\inf \mathcal{E}^h\sim
Ch^\beta$, received a lot of attention in the last
decade, via the techniques of dimension reduction and
$\Gamma$-convergence, starting with the seminal paper \cite{FJM} (see also
\cite{lew_book} and references therein).
Extending the analysis  
in \cite[Theorem 1.4]{JBL}, we get:

\begin{theorem}\label{th_scaling}
Assume that $\omega\subset\R^d$ is an open, bounded domain and let
$k\geq 1$. Denote $s = d(d+1)/k$, or $s=1$ when $k\geq d(d+1)$. Then, there holds:
\begin{itemize}
\item[(i)] if $\gamma\ge 4$, then $\inf \mathcal{E}^h\leq Ch^{\beta}$, for
  every $\beta<2+\frac{\gamma}{2}$,

\item[(ii)] if $\gamma\in \big[\frac{4}{3+s},4\big)$, then
  $\inf \mathcal{E}^h\leq Ch^{\beta}$ for every $\beta< \frac{4+\gamma (1+s)}{2+s}$,

\item[(iii)] if $\gamma\in \big(0,\frac{4}{3+s}\big)$, then $\inf
  \mathcal{E}^h\leq Ch^{\beta}$, with $\beta = 2\gamma$.
\end{itemize}
\end{theorem}

\noindent We recall that for $d=2, k=1$, the asymptotic behaviour of the
minimizing sequences to (\ref{Eh}) as $h\to 0$, is fully understood in
the scaling regime corresponding to $\beta\geq 2$ (see \cite{lew_book}).

\subsection{Organization of the paper and notation.} 
In section \ref{sec_step} we give two different constructions of the single ``step'' of the
convex integration algorithm, and recall a few auxiliary
results. The proof of Theorem \ref{thm_stage} and the ``stage''
construction is carried out in section \ref{sec_stage}, based on the
corrugation ``step'' in Lemma \ref{lem_step}. The proof 
of Theorem \ref{thm_stageK} and the corresponding ``stage''
construction based on the spirals ``step'' in Lemma \ref{lem_stepK}
is given in section \ref{sec3.5}.
The Nash-Kuiper scheme involving induction
on stages is presented in section \ref{sec4}, and Theorems
\ref{th_final} and \ref{th_finalK} are then deduced in section \ref{sec5}.
In section \ref{sec6} we discuss the Monge-Amp\'ere system and
prove Lemmas \ref{th_equiv}, \ref{th_solve}, and Theorem \ref{th_CI_weakMA}. Finally, in section
\ref{sec_appli} we prove Theorem \ref{th_scaling}.

\medskip

\noindent By $\mathbb{R}^{d\times d}_{\sym}$ we denote the space of symmetric
$d\times d$ matrices, and by  $\mathbb{R}^{d\times d}_{\sym, >}$ we
denote the cone of symmetric, positive definite $d\times d$ matrices.
The space of H\"older continuous vector fields
$\mathcal{C}^{m,\alpha}(\bar\omega,\R^k)$ consists of restrictions of
all $f\in \mathcal{C}^{m,\alpha}(\mathbb{R}^d,\R^k)$ to the closure of an open domain
$\omega\subset\R^d$. Then, the $\mathcal{C}^m(\bar\omega,\R^k)$ norm of such restriction is
denoted by $\|f\|_m$, while its H\"older norm $\mathcal{C}^{m,
  \alpha}(\bar\omega,\R^k)$ is $\|f\|_{m,\alpha}$. 
By $C>0$ we denote a universal constant which may change from line to
line, but which is independent of all parameters, unless indicated otherwise.

\section{Convex integration: the basic ``step'' and preparatory statements}\label{sec_step}

In this section, we give two different constructions of the basic building block in
the convex integration algorithm towards the proof of Theorems
\ref{th_final} and \ref{th_finalK}. The first construction below is based on
Kuiper's corrugations. A similar calculation in \cite{lewpak_MA}
had $\bar\Gamma=0$, resulting in the presence of the extra term
$-\frac{2}{\lambda} a \dbar\Gamma(\lambda t_\eta)\sym(\nabla a \otimes
\eta)$ in the right hand side of (\ref{step_err}). With that term,
the corrugation-based double induction in the proof of a stage in
section \ref{sec_stage}
would not be possible,
unless in a special situation when $d_*$ is a multiple of $k$.

\begin{lemma}\label{lem_step}
Let $v\in \mathcal{C}^2(\omega, \R^{k})$ and $w\in
\mathcal{C}^1(\omega, \R^{d})$ be two vector fields on an open domain $\omega\subset\R^d$.
Let $\eta\in\R^d$ and $E\in \R^k$ be two unit vectors and let
$\lambda>0$, $a\in \mathcal{C}^2(\omega,\R)$. We denote:
\begin{equation}\label{tripl_def}
\Gamma(t) = 2\sin t,\quad \bar\Gamma(t) = -\frac{1}{2}\cos (2t), \quad
\dbar\Gamma(t) = -\frac{1}{2}\sin (2t).
\end{equation}
Denoting further $t_\eta = \langle x,\eta\rangle$, we define:
\begin{equation}\label{defi_per}
\begin{split}
&\tilde v(x) = v(x) + \frac{1}{\lambda}a(x) \Gamma(\lambda t_\eta)E\\
& \tilde w(x) = w(x) -\frac{1}{\lambda}a(x) \Gamma(\lambda t_\eta)\nabla \langle v(x), E\rangle - 
\frac{1}{\lambda^2} a(x) \bar\Gamma(\lambda t_\eta)\nabla a(x)
+ \frac{1}{\lambda}a(x)^2 \dbar\Gamma(\lambda t_\eta)\eta.
\end{split}
\end{equation}
Then we have:
\begin{equation}\label{step_err}
\begin{split}
& \big(\frac{1}{2}(\nabla \tilde v)^T \nabla \tilde v + \sym\nabla \tilde w\big) - 
\big(\frac{1}{2}(\nabla v)^T \nabla v + \sym\nabla w\big) - a^2\eta\otimes\eta 
\\ & = -\frac{1}{\lambda} a \Gamma(\lambda t_\eta)\nabla^2 \langle v, E\rangle +
\frac{1}{\lambda^2}\big(\frac{1}{2}\Gamma(\lambda
t_\eta)^2-\bar\Gamma(\lambda t_\eta)\big) \nabla a\otimes\nabla a -
\frac{1}{\lambda^2}a \bar\Gamma(\lambda t_\eta)\nabla^2a,
\end{split}
\end{equation}
where $\frac{1}{2}\Gamma(t)^2 -\bar\Gamma(t) = 1-\frac{1}{2}\cos(2t)$.
\end{lemma}

\begin{proof}
By a direct calculation, it follows that:
$$\nabla \tilde v=\nabla v +\frac{1}{\lambda} \Gamma(\lambda
t_\eta)E\otimes \nabla a + a\Gamma'(\lambda t_\eta)E\otimes \eta,$$
which implies:
\begin{equation*}
\begin{split}
\frac{1}{2}(\nabla \tilde v)^T \nabla \tilde v - \frac{1}{2}(\nabla v)^T \nabla v
= ~ & \frac{1}{2} a^2 \Gamma'(\lambda t_\eta)^2\eta\otimes\eta +
\frac{1}{\lambda} a \Gamma'(\lambda t_\eta)\Gamma(\lambda t_\eta)
\sym\big(\nabla a\otimes\eta \big)
\\ &  + \frac{1}{2\lambda^2} \Gamma(\lambda t_\eta)^2\nabla
a\otimes\nabla a.
\\ & + \Big(a\Gamma'(\lambda t_\eta)\sym \big((\eta\otimes E)\nabla v\big)
+ \frac{1}{\lambda} \Gamma(\lambda t_\eta)\sym \big((\nabla a\otimes E)\nabla v\big)\Big).
\end{split}
\end{equation*}
Similarly:
\begin{equation*}
\begin{split}
\sym\nabla \tilde w - \sym\nabla w
= ~ & a^2 \dbar\Gamma'(\lambda t_\eta)\eta\otimes\eta
+\frac{1}{\lambda} a \big(-\bar\Gamma' (\lambda t_\eta)+2\dbar 
\Gamma (\lambda t_\eta)\big) \sym\big(\nabla a\otimes \eta\big)
\\ & - \frac{1}{\lambda} a \Gamma(\lambda t_\eta)\nabla^2\langle v,E\rangle
- \frac{1}{\lambda^2} \bar\Gamma(\lambda t_\eta)\nabla
a\otimes\nabla a - \frac{1}{\lambda^2}a \bar\Gamma(\lambda t_\eta)\nabla^2a
\\ & - \Big(a\Gamma'(\lambda t_\eta)\sym \big(\eta\otimes \nabla \langle v, E\rangle\big)
+ \frac{1}{\lambda} \Gamma(\lambda t_\eta)\sym \big(\nabla a\otimes \nabla \langle v, E\rangle\big)\Big).
\end{split}
\end{equation*}
Summing the above two identities and noting that:
\begin{equation}\label{trip_c}
\frac{1}{2}(\Gamma')^2 + \dbar\Gamma' =1 \quad \mbox{ and }\quad
\Gamma'\Gamma - \bar\Gamma' + 2\dbar\Gamma = 0,
\end{equation}
we arrive at the claimed identity (\ref{step_err}). The proof is done.
\end{proof}

\smallskip

\noindent We next observe that taking several perturbations in $v$ of the form
$\frac{1}{\lambda}a\Gamma(\lambda t_\eta) E$, and matching them with the
perturbations of $w$ as in (\ref{defi_per}), accumulates the error in
(\ref{step_err}) in a linear fashion as long as the directions $E$
are mutually orthogonal. The same calculations as above, lead to:

\begin{corollary}\label{cor_indep}
Let $v\in \mathcal{C}^2(\omega, \R^{k})$ and $w\in
\mathcal{C}^1(\omega, \R^{d})$ be two vector fields on an open domain $\omega\subset\R^d$.
Let $\{\eta_i\in\R^d\}_{i=1}^k$ be given unit vectors, and let
$\{E_i\in \R^k\}_{i=1}^k$ be some orthonormal basis of $\R^k$. Given $\{\lambda_i>0\}_{i=1}^k$ and $\{a_i\in
\mathcal{C}^1(\omega,\R)\}_{i=1}^k$, we set:
\begin{equation*}
\begin{split}
&\tilde v = v + \sum_{i=1}^k\frac{1}{\lambda_i}a_i \Gamma(\lambda_i t_{\eta_i})E_i \\ 
& \tilde w= w -\sum_{i=1}^k\frac{1}{\lambda_i}a_i \Gamma(\lambda_i t_{\eta_i})\nabla \langle v, E_i\rangle - 
\sum_{i=1}^k\frac{1}{\lambda_i^2} a_i \bar\Gamma(\lambda_i t_{\eta_i})\nabla a_i
+ \sum_{i=1}^k\frac{1}{\lambda_i}a_i^2 \dbar\Gamma(\lambda_i t_{\eta_i})\eta_i,
\end{split}
\end{equation*}
where the functions $\Gamma$, $\bar\Gamma$, $\dbar\Gamma$ and $t_\eta$
are defined as in Lemma \ref{lem_step}. Then we have:
\begin{equation*}
\begin{split}
\big(\frac{1}{2}(\nabla & \tilde v)^T \nabla \tilde v + \sym\nabla \tilde w\big) - 
\big(\frac{1}{2}(\nabla v)^T \nabla v + \sym\nabla w\big) - \sum_{i=1}^ka_i^2\eta_i\otimes\eta_i
\\  = &  - \sum_{i=1}^k\frac{1}{\lambda_i} a_i \Gamma(\lambda_i t_{\eta_i})\nabla^2 \langle v, E_i\rangle +
\sum_{i=1}^k\frac{1}{\lambda_i^2}\big(\frac{1}{2}\Gamma(\lambda_i
t_{\eta_i})^2-\bar\Gamma(\lambda_i t_{\eta_i})\big) \nabla a_i\otimes\nabla a_i
\\ & - \sum_{i=1}^k \frac{1}{\lambda_i^2}a_i \bar\Gamma(\lambda_i t_{\eta_i})\nabla^2a_i.
\end{split}
\end{equation*}
\end{corollary}

\bigskip

\noindent The second basic ``step'' construction, exhibited below and parallel to that in
Lemma \ref{lem_step}, utilizes Nash's spirals. Observe that another choice
of periodic functions satisfying (\ref{trip_c}), is:
$$\Gamma_2(t) = 2\cos t,\quad \bar\Gamma_2(t) = \frac{1}{2}\cos (2t), \quad
\dbar\Gamma_2(t) = \frac{1}{2}\sin (2t).$$ 
The above triple is conjugate to the triple in  (\ref{tripl_def}), in
the sense that superposing the two perturbations they respectively induce, in some two orthonormal 
directions $E_1$ and $E_2$, in: $\tilde v(x) = v(x) +
\frac{1}{\lambda}a(x)\Gamma(\lambda t_x)E_1 +
\frac{1}{\lambda}a(x)\Gamma_2(\lambda t_x)E_2$, together with the matching
adjustments in $\tilde w$, results in the cancellation of the error term
$\frac{1}{\lambda^2}a\bar\Gamma(\lambda t_\eta)\nabla^2a$ in the right
hand side of (\ref{step_err}). This is precisely the
reason why the Newton iteration scheme in the proof of Theorem \ref{thm_stageK} via
K\"allen's approach (more precisely: the
validity of the inductive estimate \ref{Fbound3K} in section \ref{sec3.5}), can be carried out. On the other
hand, we emphasize that this construction requires a pair of
codimensions to cancel each single rank-one defect of the form
$a\eta\otimes \eta$. 

\begin{lemma}\label{lem_stepK}
Let $\omega\subset\R^d$ be an open domain and let $k\geq 2$. 
Given $v\in \mathcal{C}^2(\omega, \R^{k})$, $w\in
\mathcal{C}^1(\omega, \R^{d})$, $a\in \mathcal{C}^2(\omega,\R)$,
$\eta\in\mathbb{S}^{d-1}$, $\lambda>0$
and two orthogonal unit vectors $E_1, E_2\in\R^k$, 
denote:
$$G(t) = \sin t, \qquad \bar G(t) = \cos t,$$ 
and denoting further $t_\eta = \langle x,\eta\rangle$, define:
\begin{equation}\label{defi_perK}
\begin{split}
&\tilde v(x) = v(x) + \frac{1}{\lambda} a(x)
\Big(G(\lambda t_{\eta})E_1 +\bar G(\lambda t_{\eta})E_2\Big)  \\
& \tilde w(x) = w(x) -\frac{1}{\lambda} a(x)\Big(G(\lambda
t_{\eta})\nabla \langle v(x), E_1\rangle +
\bar G(\lambda t_{\eta})\nabla \langle v(x), E_2\rangle\Big).
\end{split}
\end{equation}
Then, there holds:
\begin{equation}\label{step_errK}
\begin{split}
& \big(\frac{1}{2}(\nabla \tilde v)^T \nabla \tilde v + \sym\nabla \tilde w\big) - 
\big(\frac{1}{2}(\nabla v)^T \nabla v + \sym\nabla w\big) - \frac{1}{2} a^2\eta\otimes\eta
\\ & = -\frac{a}{\lambda} \Big(G(\lambda t_{\eta})\nabla^2 \langle v, E_1\rangle +
\bar G(\lambda t_{\eta})\nabla^2 \langle v, E_2\rangle\Big) +
\frac{1}{2\lambda^2} \nabla a\otimes \nabla a.
\end{split}
\end{equation}
\end{lemma}
\begin{proof}
By a direct calculation, it follows that:
$$\nabla \tilde v=\nabla v +\frac{1}{\lambda} 
\Big(G(\lambda t_{\eta})E_1+\bar G(\lambda t_{\eta}) E_2\Big) \otimes \nabla a + 
a\Big(G'(\lambda t_{\eta}) E_1 +\bar G'(\lambda  t_{\eta}) E_2\Big) \otimes \eta,$$
which implies:
\begin{equation*}
\begin{split}
& \frac{1}{2}(\nabla \tilde v)^T \nabla \tilde v - \frac{1}{2}(\nabla v)^T \nabla v
-   \frac{1}{2} a^2\eta\otimes\eta \\  & =  
a \Big(G'(\lambda t_{\eta})\sym(\nabla \langle v, E_1\rangle
\otimes \eta) +\bar G'(\lambda t_{\eta}) \sym(\nabla \langle v, E_2\rangle\otimes \eta)\Big)
\\ & \quad + \frac{1}{\lambda} \Big(G(\lambda t_{\eta})
\sym(\nabla \langle v, E_1\rangle\otimes \nabla a) + \bar G(\lambda t_{\eta})
\sym(\nabla \langle v, E_2\rangle\otimes \nabla a)\Big) 
+ \frac{1}{2\lambda^2} \nabla a\otimes\nabla a.
\end{split}
\end{equation*}
Similarly:
\begin{equation*}
\begin{split}
\sym\nabla \tilde w - \sym\nabla w
= ~ & - a \Big(G'(\lambda t_{\eta})\sym(\nabla \langle v, E_1\rangle
\otimes \eta) +\bar G'(\lambda t_{\eta}) \sym(\nabla \langle v, E_2\rangle \otimes \eta)\Big)
\\ & - \frac{1}{\lambda} \Big(G(\lambda t_{\eta})
\sym(\nabla \langle v, E_1\rangle\otimes \nabla a) + \bar G(\lambda t_{\eta})
\sym(\nabla \langle v, E_2\rangle \otimes \nabla a)\Big) 
\\ & -\frac{1}{\lambda} a \Big(G(\lambda t_{\eta})\nabla^2 \langle v, E_1\rangle +
\bar\Gamma(\lambda t_{\eta})\nabla^2 \langle v, E_2\rangle\Big).
\end{split}
\end{equation*}
Summing the above two identities we arrive at (\ref{step_errK}). The proof is done.
\end{proof}

\smallskip

\noindent Similarly as in Corollary \ref{cor_indep}, we note that:

\begin{corollary}\label{cor_indepK}
Let $\omega\subset\R^d$ be an open domain and let $k\geq 2d_*$.
Given $v\in \mathcal{C}^2(\omega, \R^{k})$, $w\in
\mathcal{C}^1(\omega, \R^{d})$,  $\{a_i\in
\mathcal{C}^2(\omega,\R)\}_{i=1}^{d_*}$, the unit vectors
$\{\eta_i\in\R^d\}_{i=1}^{d_*}$ and the frequency $\lambda>0$, set:
\begin{equation*}
\begin{split}
&\tilde v(x) = v(x) + \frac{1}{\lambda} \sum_{i=1}^{d_*} a_i(x)
\Big(G(\lambda t_{\eta_i})e_i +\bar G(\lambda t_{\eta_i})e_{d_*+i}\Big)  \\
& \tilde w(x) = w(x) -\frac{1}{\lambda} \sum_{i=1}^{d_*}
a_i(x)\Big(G(\lambda t_{\eta_i})\nabla v^i(x) +
\bar G(\lambda t_{\eta_i})\nabla v^{d_*+i} (x)\Big),
\end{split}
\end{equation*}
where the functions $G, \bar G$ and $t_\eta$ are defined in Lemma \ref{lem_stepK}.
Then, we have:
\begin{equation*}
\begin{split}
& \big(\frac{1}{2}(\nabla \tilde v)^T \nabla \tilde v + \sym\nabla \tilde w\big) - 
\big(\frac{1}{2}(\nabla v)^T \nabla v + \sym\nabla w\big) -
\frac{1}{2}\sum_{i=1}^{d_*} a_i^2\eta_i\otimes\eta_i 
\\ & = -\frac{1}{\lambda} \sum_{i=1}^{d_*}a_i \Big(G(\lambda t_{\eta_i})\nabla^2 v^i +
\bar G(\lambda t_{\eta_i})\nabla^2 v^{d_*+i}\Big) +
\frac{1}{2\lambda^2} \sum_{i=1}^{d_*} \nabla a_i\otimes \nabla a_i.
\end{split}
\end{equation*}
\end{corollary}

\bigskip

\noindent We now recall two auxiliary results from \cite{CDS}.
The first one gathers the convolution and commutator estimates \cite[Lemma 2.1]{CDS}:

\begin{lemma}\label{lem_stima}
Let $\phi\in\mathcal{C}_c^\infty(\R^d,\mathbb{R})$ be a standard
mollifier that is nonnegative, radially symmetric, supported on the
unit ball $B(0,1)\subset\R^d$ and such that $\int_{\mathbb{R}^d} \phi \dx = 1$. Denote: 
$$\phi_l (x) = \frac{1}{l^d}\phi(\frac{x}{l})\quad\mbox{ for all
}\; l\in (0,1], \;  x\in\R^d.$$
Then, for every $f,g\in\mathcal{C}^0(\mathbb{R}^d,\R)$ and every
$m,n\geq 0$ and $\beta\in (0,1]$ there holds:
\begin{align*}
& \|\nabla^{(m)}(f\ast\phi_l)\|_{0} \leq
\frac{C}{l^m}\|f\|_0,\tag*{(\theequation)$_1$}\vspace{1mm} \refstepcounter{equation} \label{stima1}\\
& \|f - f\ast\phi_l\|_0\leq C \min\big\{l^2\|\nabla^{2}f\|_0,
l\|\nabla f\|_0, {l^\beta}\|f\|_{0,\beta}\big\},\tag*{(\theequation)$_2$} \vspace{1mm} \label{stima2}\\
& \|\nabla^{(m)}\big((fg)\ast\phi_l - (f\ast\phi_l)
(g\ast\phi_l)\big)\|_0\leq {C}{l^{2- m}}\|\nabla f\|_{0}
\|\nabla g\|_{0}, \tag*{(\theequation)$_3$} \label{stima4}
\end{align*}
with a constant $C>0$ depending only on the differentiability exponent $m$.
\end{lemma}

\medskip

\noindent The next result states the decomposition of symmetric positive definite matrices
which are close to $\Id_d$, into ``primitive matrices'', as proved in \cite[Lemma 5.2]{CDS}:

\begin{lemma}\label{lem_dec_def}
Given the dimension $d\geq 1$, let $d_*$ be the dimension of the space $\R^{d\times d}_\sym$, namely:
$$d_* = \frac{d(d+1)}{2}.$$
There exist: a constant $r_0>0$, the linear maps $\{\bar a_i:\R^{d\times
  d}_{\sym}\to\R\}_{i=1}^{d_*}$, and the unit vectors $\{\eta_i\in\R^{d}\}_{i=1}^{d_*}$, such that
for all $A\in B(\Id_d,r_0)\subset \R^{d\times d}_{\sym}$, there holds:
$$A = \sum_{i=1}^{d_*} \bar a_i(A) \eta_i\otimes\eta_i \quad \mbox{ and } \quad
\bar a_i (A)\geq r_0 \; \mbox{ for all }\;  i =1\ldots d_*.$$
\end{lemma}

\section{The ``stage'' for the $\mathcal{C}^{1,\alpha}$
  approximations: a proof of Theorem \ref{thm_stage}}\label{sec_stage}

The following construction is the main technical contribution of this paper:

\medskip

\noindent {\bf Proof of Theorem \ref{thm_stage}}

The proof consists of several steps in an inductive construction below.

\smallskip

{\bf 1. (Preparing the data)} Recall that $v, w, A$ are restrictions
to $\bar\omega$ of some $v, w, A$ defined on and, without loss of
generality, compactly supported in $\R^d$. We set the mollification scale:
\begin{equation}\label{defi_l} 
l=\frac{\|\mathcal{D}\|_0^{1/2}}{M}\in (0, 1],
\end{equation}
and taking $\phi_l (x) = \frac{1}{l^d}\phi(x/l)$ as in Lemma \ref{lem_stima}, we define:
$$v_0=v\ast \phi_l,\quad w_0=w\ast \phi_l, \quad A_0=A\ast \phi_l,
\quad {\mathcal{D}}_0= A_0-\big(\frac{1}{2}(\nabla v_0)^T\nabla v_0 + \sym\nabla w_0\big).$$
From the estimates in Lemma \ref{lem_stima},  one deduces the initial bounds:
\begin{align*}
& \|v_0-v\|_1 + \|w_0-w\|_1 \leq C\|\mathcal{D}\|_{0}^{1/2},
\tag*{(\theequation)$_1$}\refstepcounter{equation} \label{pr_stima1}\\
& \|A_0-A\|_0 \leq Cl^\beta\|A\|_{0,\beta}, \tag*{(\theequation)$_2$} \label{pr_stima2}\\
& \|\nabla^{(m+1)}v_0\|_0 + \|\nabla^{(m+1)}w_0\|_0\leq
\frac{C}{l^m}\|\mathcal{D}\|_{0}^{1/2}\quad \mbox{ for all }\; m\geq
1, \tag*{(\theequation)$_3$} \label{pr_stima3}\\
& \|\nabla^{(m)} \mathcal{D}_0\|_0\leq
\frac{C}{l^m}\|\mathcal{D}\|_{0} \quad \mbox{ for all }\; m\geq 0. \tag*{(\theequation)$_4$}\label{pr_stima4}
\end{align*}
Indeed,  \ref{pr_stima2} follows directly from \ref{stima2}, and 
\ref{pr_stima1} similarly follows by applying \ref{stima2} to
$v$, $\nabla v$, $w$, $\nabla w$ and noting that,  in view of (\ref{defi_l}) we have:
\begin{equation}\label{stupido}
l\|v\|_2+ l\|w\|_2 \leq 2\|\mathcal{D}\|_0^{1/2}.
\end{equation}
Further, \ref{pr_stima3} follows by applying \ref{stima1} to
$\nabla^2v$ and $\nabla^2w$ with the differentiability exponent $m-1$ and
again taking into account (\ref{stupido}). To check \ref{pr_stima4}, we write:
$$\mathcal{D}_0 = \mathcal{D}\ast \phi_l -\frac{1}{2}\big((\nabla
v_0)^T\nabla v_0 - ((\nabla v)^T\nabla v)\ast\phi_l\big),$$ 
and apply \ref{stima1} to $\mathcal{D}$, and \ref{stima4} to $(\nabla
v)^T$ and $\nabla v$, where the final bound is due to (\ref{stupido}). 

\medskip

{\bf 2. (Induction definition: frequencies)}  We now inductively define the main
constants, frequencies and corrections in the construction of ($\tilde
v,\tilde w)$ from $(v,w)$. First, we write the least common multiple
of the auxiliary dimension $d_*$ and the codimension $k$, as follows:
\begin{equation}\label{lcm_def}
N=lcm(d_*,k) = Sd_* = Jk, \qquad S,J\geq 1.
\end{equation}
Then, we set the initial perturbation frequencies as:
$$\lambda_0 = \frac{1}{l}, \qquad \lambda_1=\lambda = \frac{\sigma^{1/S}}{l}.$$
For every $i=2\ldots N$ we define $\lambda_i\geq 1$ according to the
mutually exclusive cases in:
\begin{equation*}
\lambda_i = \lambda_{i-1} \cdot\left\{\begin{array}{ll} (\lambda l) & \mbox{ if }\; k\mid (i-1), \\  
(\lambda l)^{1/2} & \mbox{ if }\; d_*\mid (i-1), \\  1 & \mbox{ otherwise. }\\  
\end{array}\right.
\end{equation*}
It follows that for all $j=0\ldots J-1$ and $s=0\ldots S-1$ there holds: 
\begin{equation}\label{count_lam}
\lambda_i l = (\lambda l)^{1+j+s/2 }\quad \mbox{ for all }\; i\in (jk, (j+1)k]\cap (sd_*, (s+1)d_*].
\end{equation}

\medskip

{\bf 3. (Induction definition: decomposition of deficits)}  First, let
$\{\eta_\delta\in\R^d\}_{\delta=1}^{d_*}$ be the unit vectors as in
Lemma \ref{lem_dec_def}. For all $s=0\ldots 
S-1$ we define constants $\tilde C_s$ and perturbation amplitudes vector
$a^s = [a_\delta^s]_{\delta=1}^{d_*}\in\mathcal{C}^\infty(\bar\omega,\R^{d_*})$ by:
\begin{equation*}
\begin{split}
& \tilde C_s = \frac{2}{r_0}\Big(\frac{1}{(\lambda l)^s} \|\mathcal{D}\|_0 + \|\mathcal{D}_s\|_0\Big),\\
& a_\delta^s(x) = \Big(\tilde C_s\bar a_\delta\big(\Id_d + \frac{1}{\tilde C_s}\mathcal{D}_s(x)\big)\Big)^{1/2}
\quad \mbox{ for all }\; \delta=1\ldots d_*, \;\; x\in\bar\omega.
\end{split}
\end{equation*}
Above, $r_0>0$ and the maps ${\bar a}_\delta$ are as in Lemma
\ref{lem_dec_def}, so our definition is correctly posed because
$\Id_d + \frac{1}{\tilde C_s}\mathcal{D}_s (x)\in
B(\Id_d,r_0)\subset\R^{d\times d}_\sym$ for all $x\in \bar\omega$. As
$\tilde C_s \Id_d + \mathcal{D}_s = \tilde C_s \big(\Id_d + 
\frac{1}{\tilde C_s}\mathcal{D}_s \big)$, we get: 
\begin{equation}\label{low_bd_as}
\tilde C_s \Id_d + \mathcal{D}_s = \sum_{\delta=1}^{d_*}
(a_\delta^s)^2\eta_\delta\otimes\eta_\delta \quad\mbox{ and }\quad (a_\delta^s)^2\geq
r_0\tilde C_s \;\mbox{ in } \;\bar\omega,\;\mbox{ for all } \;\delta=1\ldots d_*.
\end{equation}

Since $\{\eta_\delta\otimes\eta_\delta\}_{\delta=1}^{d_*}$ is
a basis of the linear space $\R^{d\times d}_\sym$, we obtain:
\begin{equation}\label{as0}
\|a^s\|_0\leq C\| \tilde C_s \Id_d + \mathcal{D}_s\|_0^{1/2}\leq C \tilde C_s^{1/2}.
\end{equation}
We also right away observe that, by the  Fa\'a di Bruno formula, there
holds, for $m\geq 1$:
\begin{equation*}
\|\nabla^{(m)}a_\delta^s\|_0\leq C \Big\|\sum_{p_1+2p_2+\ldots
  mp_m=m}|a_\delta^s|^{2(1/2-p_1-\ldots -p_m)}\prod_{t=1}^m \big|\nabla^{(t)}|a_\delta^s|^2\big|^{p_t}\Big\|_0
\quad\mbox{ for all }\; \delta=1\ldots d_*.
\end{equation*}
Using the lower bound in (\ref{low_bd_as}) and the linearity of $\bar
a_\delta$ in Lemma \ref{lem_dec_def}, we further get: 
\begin{equation}\label{asm}
\begin{split}
\|\nabla^{(m)}a^s\|_0 & \leq C \sum_{p_1+2p_2+\ldots
  mp_m=m}\frac{1}{\tilde C_s^{(p_1+\ldots+p_m)-1/2}}\prod_{t=1}^m
\big\|\nabla^{(t)}\mathcal{D}_s\big\|_0^{p_t} \\ & \leq {C}{\tilde
  C_s^{1/2}} \sum_{p_1+2p_2+\ldots mp_m=m}\prod_{t=1}^m
\Big(\frac{\|\nabla^{(t)}\mathcal{D}_s\|_0}{\tilde C_s}\Big)^{p_t}.
\end{split}
\end{equation}

In particular, for $s=0$ and any $\delta=1\ldots d_*$, the bounds
(\ref{as0}), (\ref{asm}) and \ref{pr_stima4} yield:
\begin{equation}\label{induC_0} 
\tilde C_0 \leq C\|\mathcal{D}\|_0 \quad\mbox{ and }\quad 
\|\nabla^{(m)}a^0\|_0 \leq
\frac{C}{l^m}\|\mathcal{D}\|_{0}^{1/2}\quad \mbox{ for all }\; m\geq 0.
\end{equation}

\begin{figure}[htbp]
\centering
\includegraphics[scale=0.55]{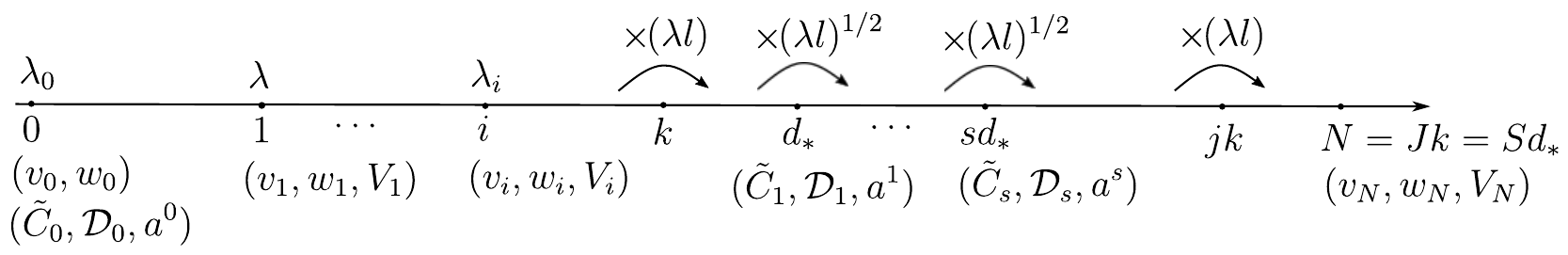}
\caption{{Progression of frequencies $\lambda_i$ and other
    intermediary quantities defined at integers $i=1\ldots N$, where $N= lcm(k, d_*)$.}}
\label{fig1}
\end{figure}

\medskip

{\bf 4. (Induction definition: perturbations)} For each $i=1\ldots N$ we may uniquely write:
\begin{equation}\label{isj}
\begin{split}
i=jk+\gamma = sd_* + \delta \quad \mbox{ with } \quad  & j=0\ldots J-1,
\quad \gamma=1\ldots k, \\ &  s=0\ldots S-1, \quad \delta=1\ldots d_*.
\end{split}
\end{equation}
Define $v_{i}\in\mathcal{C}^\infty(\bar\omega,\R^k)$ and
$w_i\in\mathcal{C}^\infty(\bar\omega, \R^d)$ according to the ``step''
construction in Lemma \ref{lem_step}, involving the periodic profile functions
$\Gamma, \bar\Gamma, \dbar\Gamma$ and the notation $t_\eta=\langle x, \eta\rangle$:
\begin{equation*}
\begin{split}
& v_i(x) = v_{i-1}(x) + \frac{1}{\lambda_i} a_\delta^s(x)\Gamma(\lambda_i t_{\eta_\delta})e_\gamma,\\
& w_i(x) = w_{i-1}(x) - \frac{1}{\lambda_i} a_\delta^s(x)\Gamma(\lambda_it_{\eta_\delta})\nabla v_{i-1}^\gamma 
-  \frac{1}{\lambda_i^2} a_\delta^s(x)\bar\Gamma(\lambda_it_{\eta_\delta})\nabla a_{\delta}^s
+  \frac{1}{\lambda_i} a_\delta^s(x)^2\dbar\Gamma(\lambda_it_{\eta_\delta})\eta_\delta.
\end{split}
\end{equation*}
We observe that by construction of $v_i$, the second term in $w_{i}$ can be rewritten as follows:
\begin{equation}\label{w_simp}
\frac{1}{\lambda_i} a_\delta^s(x)\Gamma(\lambda_it_{\eta_\delta})\nabla v_{i-1}^\gamma 
= \frac{1}{\lambda_i} a_\delta^s(x)\Gamma(\lambda_it_{\eta_\delta})\nabla v_{jk}^\gamma. 
\end{equation}
We eventually define:
\begin{equation}\label{vw_fin}
\tilde v = v_N,\qquad \tilde w = w_N-\sum_{s=0}^{S-1}\tilde C_sx.
\end{equation}

\medskip

{\bf 5. (Induction definition: deficits)} For each $i=1\ldots N$, we define the partial deficit:
$$ V_i = \big(\frac{1}{2}(\nabla v_i)^T\nabla v_i+\sym\nabla
w_i\big) - \big(\frac{1}{2}(\nabla v_{i-1})^T\nabla v_{i-1}+\sym\nabla w_{i-1}\big),$$
and for each $s=1\ldots S$ we define the combined deficit
$\mathcal{D}_s\in\mathcal{C}^\infty(\bar\omega, \R^{d\times d}_\sym)$ in:
\begin{equation*}
\begin{split}
\mathcal{D}_s  = & \; -\big(\frac{1}{2}(\nabla v_{sd_*})^T\nabla v_{sd_*}+\sym\nabla
w_{sd_*}\big) + \big(\frac{1}{2}(\nabla v_{(s-1)d_*})^T\nabla v_{(s-1)d_*}+\sym\nabla w_{(s-1)d_*}\big)
\\ & +\sum_{\delta=1}^{d_*}(a_\delta^{s-1})^2\eta_\delta\otimes\eta_\delta
= - \sum_{i=(s-1)d_*+1}^{sd_*} \Big(V_i - (a_\delta^{s-1})^2\eta_\delta\otimes \eta_\delta\Big),
\end{split}
\end{equation*}
where in components of the last sum we used the convention (\ref{isj}), setting $\delta=\delta(i)=1\ldots d_*$.
By Lemma \ref{lem_step} and noting (\ref{w_simp}), for each $i=(s-1)d_* \ldots {sd_*}$ as above, we get:
\begin{equation}\label{prep_defi_s}
\begin{split}
V_i - (a_\delta^{s-1})^2\eta_\delta\otimes \eta_\delta
 = & - \frac{1}{\lambda_i} a_\delta^{s-1} \Gamma(\lambda_i t_{\eta_\delta})\nabla^2 v_{jk}^\gamma
- \frac{1}{\lambda_i^2}a^{s-1}_\delta \bar\Gamma(\lambda_i t_{\eta_\delta})\nabla^2a^{s-1}_\delta
\\ & + \frac{1}{\lambda_i^2}\big(\frac{1}{2}\Gamma(\lambda_i
t_{\eta_\delta})^2-\bar\Gamma(\lambda_i t_{\eta_\delta})\big) \nabla a^{s-1}_\delta\otimes\nabla a^{s-1}_\delta,
\end{split}
\end{equation}
where $j=0\ldots J-1$ is again set according to (\ref{isj}).

\medskip

{\bf 6. (Inductive estimates)} In steps 7-10 below we will prove the following estimates:
\begin{align*}
& \hspace{-3mm} \left. \begin{array}{l} \|v_i - v_{i-1}\|_1\leq
    C\|\mathcal{D}\|_0^{1/2} \vspace{1mm} \\ \|w_i -
w_{i-1}\|_1\leq C\|\mathcal{D}\|_0^{1/2}(1+\|\nabla v\|_0) \end{array}\right\}\quad\mbox{
  for all }\; i=1\ldots N, \tag*{(\theequation)$_1$}\refstepcounter{equation} \label{Fbound1}\medskip\\
& \hspace{-3mm} \left. \begin{array}{l}
\displaystyle{ \|\nabla^{(m+1)}v_{kj}\|_0\leq C\frac{\lambda_{kj}^{m-1}}{l} (\lambda l)^j\|\mathcal{D}\|_0^{1/2} }
\vspace{1mm} \\  \displaystyle{\|\nabla^{(m+1)}w_{kj}\|_0\leq C\frac{\lambda_{kj}^{m-1}}{l} (\lambda
l)^j\|\mathcal{D}\|_0^{1/2} (1+\|\nabla v\|_0)}
\end{array}\right\} \quad\mbox{ for all }\; j=0\ldots J, \; \; m\geq
1, \tag*{(\theequation)$_2$}\label{Fbound2} \medskip\\
& \tilde C_s\leq \frac{C}{(\lambda l)^s}\|\mathcal{D}\|_0, \; \quad
\|\nabla^{(m)}\mathcal{D}_s\|_0\leq C\frac{\lambda_{sd_*}^m}{(\lambda
  l)^s}\|\mathcal{D}\|_0 \quad\mbox{ for all }\; s=0\ldots S, \;\;
m\geq 0, \tag*{(\theequation)$_3$}\label{Fbound3} \medskip\\
& \|\nabla^{(m)}a^s\|_0\leq C\frac{\lambda_{sd_*}^m}{(\lambda
  l)^{s/2}}\|\mathcal{D}\|_0^{1/2} \quad\mbox{ for all }\; s=0\ldots
S-1, \;\; m\geq 0. \tag*{(\theequation)$_4$}\label{Fbound4} 
\end{align*}
We observe that all the bounds are already valid at their lowest counter values:
by \ref{pr_stima3} there holds \ref{Fbound2} for $j=0$,
the first bound in \ref{Fbound3} and the bound in \ref{Fbound4} at
$s=0$ have been established in (\ref{induC_0}), while the second bound
in \ref{Fbound3} at $s=0$ is exactly \ref{pr_stima4}. To show
\ref{Fbound1} at $i=1$, we use (\ref{induC_0}) and \ref{pr_stima3} in:
\begin{equation*}
\begin{split}
& \|v_1-v_0\|_1\leq C\Big(\|a^0\|_0 +
\frac{\|\nabla a^0\|_0}{\lambda}\Big) \leq C\|\mathcal{D}\|^{1/2}_0
\big(1+\frac{1}{\lambda l}\big)\leq C\|\mathcal{D}\|^{1/2}_0,\\
& \|w_1-w_0\|_1\leq C\Big(\|a^0\|_0 \|\nabla v_0\|_0 + \|a^0\|_0^2 + 
\frac{\|\nabla a^0\|_0^2+ \|a^0\|_0 \|\nabla^2 a^0\|_0}{\lambda^2} \\
& \qquad\qquad \qquad\qquad + 
\frac{\|a^0\|_0 \|\nabla a^0\|_0 +\|\nabla a^0\|_0 \|\nabla v_0\|_0+\|a^0\|_0 \|\nabla^2 v_0\|_0}{\lambda} 
\Big) \\ & \qquad\qquad\quad \leq C\|\mathcal{D}\|_0^{1/2}(1+\|\nabla v\|_0 ),
\end{split}
\end{equation*}
because $\lambda l\geq 1$ and $\|\nabla v_0\|_0\leq \|\nabla v\|_0 +
C\|\mathcal{D}\|^{1/2}\leq C(1+ \|\nabla v\|_0 )$ from \ref{pr_stima1}.

\bigskip

{\bf 7. (Proof of estimate \ref{Fbound1})} For $i\in(1, N]$, we write:
$$i\in (jk, (j+1)k]\cap (sd_*, (s+1)d_*]$$ 
with $j,s$ as in (\ref{isj}). By \ref{Fbound4}, we get:
\begin{equation*}
\|v_i-v_{i-1}\|_1\leq C\Big(\|a^s\|_0 +
\frac{\|\nabla a^s\|_0}{\lambda_i}\Big) \leq \frac{C}{(\lambda
  l)^{s/2}}\|\mathcal{D}\|^{1/2}_0\big(1+
\frac{\lambda_{sd_*}}{\lambda_i }\big)\leq C\|\mathcal{D}\|^{1/2}_0,
\end{equation*}
where we used that $\lambda l\geq 1$ and $\lambda_{sd_*} l \leq
\lambda_i l$, due to  $i>sd_*$. The bound for the $w$-increment follows by
\ref{Fbound2} at $m=1$, \ref{Fbound4}, \ref{Fbound1} and \ref{pr_stima1}:
\begin{equation*}
\begin{split}
\|w_i-w_{i-1}\|_1 &\leq C\Big(\|a^s\|_0 \|\nabla v_{jk}\|_0 +
\|a^s\|_0^2 + \frac{\|\nabla a^s\|_0^2 + \|a^s\|_0 \|\nabla^2
  a^s\|_0}{\lambda_i^2} \\ & \quad \qquad + \frac{\|\nabla a^s\|_0\|\nabla v_{jk}\|_0
  +\|a^s\|_0\|\nabla^2 v_{jk}\|_0 + \|a^s\|_0 \|\nabla a^s\|_0}{\lambda_i}
\Big) \\ & \leq \frac{C}{(\lambda l)^{s/2}}\|\mathcal{D}\|^{1/2}_0\big(1+
\frac{\lambda_{sd_*}}{\lambda_i }+\frac{\lambda_{sd_*}^2}{\lambda_i^2}
  +\frac{(\lambda l)^j}{\lambda_il} \big)\big(1+\|\nabla v\|_0\big)\leq
C\|\mathcal{D}\|^{1/2}_0 \big(1+\|\nabla v\|_0\big),
\end{split}
\end{equation*}
where again we used $\lambda_{sd_*} l \leq \lambda_i l$ due to
$i>sd_*$, and $(\lambda l)^j\leq \lambda_il$ due to $i>jk$.

\medskip

{\bf 8. (Proof of estimate \ref{Fbound2})} Let $i=1\ldots N$ and $m\geq 1$. Write: 
$$i\in ((j-1)k, jk]\cap (sd_*, (s+1)d_*]$$ 
with $j=1\ldots J$, $s=0\ldots S-1$. Then:
\begin{equation*}
\begin{split}
\|\nabla^{(m+1)}(v_i-v_{i-1})\|_0 & \leq C
\sum_{p+q=m+1}\lambda_i^{p-1}\|\nabla^{(q)}a^s\|_0 \leq C\lambda_i^{m-1}
\sum_{q=0}^{m+1}\frac{\lambda_{sd_*}^q \lambda_i \|\mathcal{D}\|_0^{1/2}}{\lambda_i^q (\lambda l)^{s/2}}
\\ & \leq C \lambda_i^{m-1} \frac{\lambda_i}{(\lambda l)^{s/2} }
\|\mathcal{D}\|_0^{1/2} = C \lambda_i^{m-1}
\frac{(\lambda l)^j}{l} \|\mathcal{D}\|_0^{1/2}
\end{split}
\end{equation*}
because $\lambda_{sd_*} \leq \lambda_i $ due to $i>sd_*$, and in fact from (\ref{count_lam}):
$$\lambda_i=\frac{(\lambda l)^{j+s/2}}{l}.$$ 
The above justifies:
$$\|\nabla^{(m+1)}(v_{kj} - v_{(j-1)k})\|_0 \leq \sum_{i=(j-1)k+1}^{jk}\|\nabla^{(m+1)}(v_i-v_{i-1})\|_0
\leq C\lambda_{jk}^{m-1}\frac{(\lambda l)^j}{l} \|\mathcal{D}\|_0^{1/2},$$
since $i\mapsto \lambda_i$ is a nondecreasing function. Further, by
\ref{pr_stima3} we get:
$$\|\nabla^{(m+1)}v_0\|_0\leq \frac{C}{l^m}\|\mathcal{D}\|_0^{1/2} =
C\frac{\lambda_0^{m-1}}{l} \|\mathcal{D}\|_0^{1/2}\leq
C\lambda_{jk}^{m-1}\frac{(\lambda l)^j}{l} \|\mathcal{D}\|_0^{1/2}.$$ 
The above two bounds prove the first statement in \ref{Fbound2}. 

\smallskip

Towards proving the second bound, we note that the increment in $w$ is estimated:
\begin{equation}\label{wdif_pom}
\begin{split}
\|\nabla^{(m+1)}&(w_i-w_{i-1})\|_0  \leq  C
\sum_{p+q+t=m+1}\lambda_i^{p-1}\|\nabla^{(q)}a^s\|_0 \|\nabla^{(t+1)}v_{(j-1)k}\|_0
\\ & + C \sum_{p+q+t=m+1}\Big(\lambda_i^{p-2}\|\nabla^{(q)}a^s\|_0 \|\nabla^{(t+1)}a^s\|_0 +
\lambda_i^{p-1}\|\nabla^{(q)}a^s\|_0 \|\nabla^{(t)}a^s\|_0\Big). 
\end{split}
\end{equation}
We split the first sum in the right hand side above into cases $t=0$
and $t\geq 1$, so that by \ref{Fbound4} and \ref{Fbound2}, together
with \ref{Fbound1} and \ref{pr_stima1}:
\begin{equation*}
\begin{split}
& \sum_{p+q+t=m+1}\lambda_i^{p-1}\|\nabla^{(q)}a^s\|_0 \|\nabla^{(t+1)}v_{(j-1)k}\|_0
\\ & \qquad \quad\leq \|\nabla v_{(j-1)k}\|_0\sum_{p+q=m+1}\lambda_i^{p-1}\|\nabla^{(q)}a^s\|_0 
+ \sum_{p+q+t=m}\lambda_i^{p-1}\|\nabla^{(q)}a^s\|_0 \|\nabla^{(t+2)}v_{(j-1)k}\|_0
\\ &\qquad \quad \leq C\lambda_i^m \big(1+\|\nabla v\|_0\big)\sum_{q=0}^{m+1}
\frac{\lambda_{sd_*}^q \|\mathcal{D}\|_0^{1/2}}{\lambda_i^q (\lambda l)^{s/2}}
+ C\lambda_i^m \sum_{p+q+t=m}
\frac{\lambda_{sd_*}^q} {\lambda_i^q} \frac{\lambda_{(j-1)k}^t}{\lambda_i^t}\frac{
(\lambda l)^{j-1}\|\mathcal{D}\|_0}{(\lambda_i l)(\lambda l)^{s/2}} \\ &
\qquad\quad \leq C\lambda_i^{m} \frac{\|\mathcal{D}\|_0^{1/2}}{(\lambda l)^{s/2}}
\big(1+\|\nabla v\|_0\big) + C\lambda_i^{m-1} \|\mathcal{D}\|_0
\frac{(\lambda l)^{j-1-s/2}}{l}
\end{split}
\end{equation*}
where in the last bound we used the fact that $\lambda_{sd_*} \leq
\lambda_i $ due to $i>sd_*$, and $\lambda_{(j-1)k} \leq \lambda_i $ due to $i>(j-1)k$.
The second term in (\ref{wdif_pom}) is similarly estimated:
\begin{equation*}
\begin{split}
& \sum_{p+q+t=m+1}\Big(\lambda_i^{p-2}\|\nabla^{(q)}a^s\|_0 \|\nabla^{(t+1)}a^s\|_0 +
\lambda_i^{p-1}\|\nabla^{(q)}a^s\|_0 \|\nabla^{(t)}a^s\|_0\Big)
\\ &\qquad \quad \leq C\lambda_i^m \sum_{p+q+t=m+1}
\Big(\frac{\lambda_{sd_*}^{q+t+1}} {\lambda_i^{q+t+1}} \frac{\|\mathcal{D}\|_0}{(\lambda l)^{s}} 
+ \frac{\lambda_{sd_*}^{q+t}} {\lambda_i^{q+t}} \frac{\|\mathcal{D}\|_0}{(\lambda l)^{s}} \Big)\leq
C\lambda_i^m \frac{\|\mathcal{D}\|_0}{(\lambda l)^s}.
\end{split}
\end{equation*}
Summing the last two displayed formulas, gives in view of  (\ref{wdif_pom}):
\begin{equation*}
\begin{split}
\|\nabla^{(m+1)}(w_i-w_{i-1})\|_0  & \leq C\lambda_i^{m-1}\|\mathcal{D}\|_0^{1/2}
\Big(\frac{\lambda_i}{(\lambda l)^{s/2}} + \frac{(\lambda l)^{j-1-s/2}}{l}\Big)
\big(1+\|\nabla v\|_0\big) \\ & \leq  C\lambda_i^{m-1}\frac{(\lambda l)^j}{l}\|\mathcal{D}\|_0^{1/2}
\big(1+\|\nabla v\|_0\big).  
\end{split}
\end{equation*}
The above implies the second statement in \ref{Fbound2}, in view of \ref{pr_stima3} resulting in: 
$$\|\nabla^{(m+1)}w_0\|_0 \leq \frac{C}{l^m}\|\mathcal{D}\|_0^{1/2} \leq
C\lambda_{jk}^{m-1}\frac{(\lambda l)^j}{l}  \|\mathcal{D}\|_0^{1/2},$$
and since $i\mapsto \lambda_i$ is a nondecreasing function, which yields:
\begin{equation*}
\begin{split}
\|\nabla^{(m+1)}(w_{kj} - w_{(k-1)j})\|_0 & \leq \sum_{i=(k-1)j+1}^{kj}\|\nabla^{(m+1)}(w_i-w_{i-1})\|_0
\\ & \leq C\lambda_{jk}^{m-1}\frac{(\lambda l)^j}{l} \|\mathcal{D}\|_0^{1/2}\big(1+\|\nabla v\|_0\big).
\end{split}
\end{equation*}

\medskip

{\bf 9. (Proof of estimate \ref{Fbound3})} Let $i=1\ldots N$ and $m\geq 0$. Write: 
$$i\in (jk, (j+1)k]\cap ((s-1)d_*, sd_*]$$ 
with $j=0\ldots J-1$, $s=1\ldots S$. Denoting $\delta = i-(s-1)d_*$,
we use \ref{Fbound2}, \ref{Fbound4} in (\ref{prep_defi_s}):
\begin{equation*}
\begin{split}
\big\|\nabla^{(m)}& \big(V_i -(a_\delta^{s-1})^2\eta_\delta\otimes\eta_\delta\big)\big\|_0 \leq 
C \sum_{p+q+t=m} \lambda_i^{p-1}\|\nabla^{(q)}a^{s-1}\|_0\|\nabla^{(t+2)}v_{jk}\|_0 
\\ & \quad + C \sum_{p+q+t=m} \lambda_i^{p-2}\Big(\|\nabla^{(q+1)}a^{s-1}\|_0\|\nabla^{(t+1)}a^{s-1}\|_0 
+ \|\nabla^{(q)}a^{s-1}\|_0\|\nabla^{(t+2)}a^{s-1}\|_0 \Big) \\ & \leq
C\lambda_i^m \|\mathcal{D}\|_0 \Big(\sum_{p+q+t=m}
\frac{\lambda_{(s-1)d_*}^q \lambda_{jk}^t}{\lambda_i^{q+t}}\frac{(\lambda l)^{j-(s-1)/2}}{\lambda_i l}
+ \sum_{p+q+t=m}
\frac{\lambda_{(s-1)d_*}^{q+t+2} }{\lambda_i^{q+t+2}}\frac{1}{(\lambda  l)^{s-1}} \Big).
\end{split}
\end{equation*}
Since $\lambda_{(s-1)d_*}\leq \lambda_i$ by $i>(s-1)d_*$, and
$\lambda_{jk}\leq \lambda_i$ by $i>jk$, we simplify the above estimate:  
\begin{equation*}
\begin{split}
\big\|\nabla^{(m)}\big(V_i
-(a_\delta^{s-1})^2\eta_\delta\otimes\eta_\delta\big)\big\|_0 & \leq 
C\lambda_i^m \|\mathcal{D}\|_0 \Big(\frac{(\lambda l)^{j-(s-1)/2}}{\lambda_i l}
+  \frac{\lambda_{(s-1)d_*}^{2} }{\lambda_i^{2}}\frac{1}{(\lambda l)^{s-1}} \Big)
\\ & \leq C\frac{\lambda_i^m}{(\lambda l)^{s}} \|\mathcal{D}\|_0, 
\end{split}
\end{equation*}
where the last bound follows since $\lambda_i\geq \lambda_{(s-1)d_*}
(\lambda l)^{1/2}$ by $i>(s-1)d_*$, and since (\ref{count_lam}) yields:
$$\lambda_i = \lambda (\lambda l)^{j+(s-1)/2}.$$
Note that having the second power of the quotient
$\lambda_{(s-1)d_*}/\lambda_i$ was essential to provide the missing
multiplier $\frac{1}{\lambda l}$ in order for
both considered error terms to have  the right order $C\frac{\lambda_i^m}{(\lambda l)^{s}}
\|\mathcal{D}\|_0$. It is precisely at this point where we are using
the step construction in Lemma \ref{lem_step}, and where the
previous construction in \cite[Lemma 2.2]{lewpak_MA} would not be sufficient.

Consequently, summing the partial deficits in $\mathcal{D}_s$, we
obtain the second bound in \ref{Fbound3}:
$$\|\nabla^{(m)}\mathcal{D}_s\|_0\leq C
\sum_{i=(s-1)d_*+1}^{sd_*}\frac{\lambda_i^m}{(\lambda l)^s}
\|\mathcal{D}\|_0 \leq  C \frac{\lambda_{sd_*}^m}{(\lambda l)^{s}} \|\mathcal{D}\|_0,$$
as $\lambda_i\leq\lambda_{sd_*}$ for all $i\leq sd_*$.
The first bound in \ref{Fbound3} is now also immediate:
$$\tilde C_s \leq C\Big(\frac{\|\mathcal{D}\|_0}{(\lambda l)^s} +
\|\mathcal{D}_s\|_0\Big)\leq \frac{C}{(\lambda l)^{s}} \|\mathcal{D}\|_0.$$

\medskip

{\bf 10. (Proof of estimate \ref{Fbound4})} Let $s=1\ldots S-1$. From
(\ref{as0}) and the first bound in \ref{Fbound3}, we readily deduce
\ref{Fbound4} at $m=0$:
$$\|a^s\|_0\leq \frac{C}{(\lambda l)^{s/2}}\|\mathcal{D}\|_0^{1/2}.$$
For $m\geq 1$, we use the preparatory bound (\ref{asm}) in which we
take account of \ref{Fbound3} and \ref{Fbound3}:
\begin{equation*}
\|\nabla^{(m)}a^s\|_0\leq \frac{C}{(\lambda
  l)^{s/2}}\|\mathcal{D}\|_0^{1/2} \sum_{p_1+2p_2+\ldots +mp_m=m}
\prod_{t=1}^m \Big(\lambda_{sd_*}^{tp_t}\Big(\frac{\|\mathcal{D}\|_0}{(\lambda l)^s
  \tilde C_s}\Big)^{p_t}\Big)\leq C\frac{\lambda_{sd_*}^m}{(\lambda l)^{s/2}}\|\mathcal{D}\|_0^{1/2}, 
\end{equation*}
in virtue of having $\frac{\|\mathcal{D}\|_0}{(\lambda l)^s \tilde
  C_s}\leq C$. This completes the proof of all the inductive estimates.

\medskip

{\bf 11. (End of proof)} 
We now show that \ref{Fbound1} - \ref{Fbound4} imply the bounds
claimed in the Theorem. Recall (\ref{vw_fin}), and use
\ref{Fbound1}, \ref{Fbound3} and \ref{pr_stima1} to conclude \ref{Abound1}:
\begin{equation*}
\begin{split}
& \|\tilde v- v\|_1\leq \|v_0-v\|_1 + \sum_{i=1}^N \|v_i - v_{i-1}\|_1\leq C \|\mathcal{D}\|_0^{1/2} ,\\
& \|\tilde w - w\|_1\leq \|w_0-w\|_1 + \sum_{i=1}^N \|w_i -
w_{i-1}\|_1 + C \sum_{s=0}^{S-1}\tilde C_s\leq  C \|\mathcal{D}\|_0^{1/2} (1+\|\nabla v\|_0).
\end{split}
\end{equation*}
By \ref{Fbound2} with $m=1$, there follows \ref{Abound2}:
\begin{equation*}
\begin{split}
& \|\nabla^2\tilde v\|_0= \|\nabla^2 v_N\|_0\leq C\frac{(\lambda
  l)^{J}}{l} \|\mathcal{D}\|_0^{1/2} = CM(\lambda l)^J =
CM\sigma^{J/S} = CM\sigma^{d_*/k},\\
&\|\nabla^2\tilde w\|_0= \|\nabla^2 w_N\|_0\leq C\frac{(\lambda
  l)^{J}}{l} \|\mathcal{D}\|_0^{1/2} (1+\|\nabla v\|_0)= CM\sigma^{d_*/k} (1+\|\nabla v\|_0),
\end{split}
\end{equation*}
where we used the definition $\sigma = (\lambda l)^S$ and the fact that:
$$\frac{J}{S} = \frac{J}{N}\cdot\frac{N}{S} = \frac{d_*}{k}.$$
Finally, \ref{pr_stima2}, and \ref{Fbound3} applied with $m=0$ yield \ref{Abound3}:
\begin{equation*}
\begin{split}
\|\tilde{\mathcal{D}}\|_0 = & ~ \|(A-A_0) -\mathcal{D}_S\|_0 \leq
\|A-A_0\|_0 + \|\mathcal{D}_S\|_0 \\ \leq & ~ C\Big(l^\beta\|A\|_{0,\beta} +
\frac{\|\mathcal{D}\|_0}{(\lambda l)^S}\Big) \\ = & ~  
C\Big(\frac{\|A\|_{0,\beta}}{M^\beta}  \|\mathcal{D}\|_0^{\beta/2} +
\frac{\|\mathcal{D}\|_0}{\sigma}\Big),
\end{split}
\end{equation*}
in view of the following direct decomposition:
\begin{equation*}
\begin{split}
\tilde{\mathcal{D}} & = (A-A_0) +\mathcal{D}_0 -
\Big(\big(\frac{1}{2}(\nabla\tilde v)^T\nabla \tilde v + \sym\nabla \tilde w\big)
- \big(\frac{1}{2}(\nabla v_0)^T\nabla v_0 + \sym\nabla w_0\big)\Big) \\ & = 
(A-A_0) +\mathcal{D}_0 +\sum_{s=0}^{S-1}\tilde C_s\Id_d - \sum_{s=1}^{S}\sum_{i=(s-1)d_*+1}^{sd_*}V_i
\\ & = (A-A_0) +\sum_{s=0}^{S-1}\tilde C_s\Id_d  + \sum_{s=0}^S\mathcal{D}_s -
\sum_{s=1}^{S}\sum_{\delta=1}^{d_*}(a_\delta^{s-1})^2\eta_\delta\otimes
\eta_\delta \\ & = (A-A_0) +\mathcal{D}_S
\end{split}
\end{equation*}
The proof is done.
\endproof

\section{The ``stage'' using K\"allen's approach: a proof of Theorem
  \ref{thm_stageK}}\label{sec3.5} 

In this section, we prove flexibility of (\ref{VK}) up to regularity
$\mathcal{C}^{1,1}$, provided that $k\geq 2d_*$.

\medskip

\noindent {\bf Proof of Theorem \ref{thm_stageK}}

\smallskip

{\bf 1. (Preparing the data)} We set the mollification scale $l$ and the frequency $\lambda$:
\begin{equation}\label{defi_lK} 
l=\frac{\|\mathcal{D}\|_0^{1/2}}{M}\in (0, 1], \qquad \lambda =
\frac{\sigma^{1/N}}{l} >1 \quad \mbox{where } \; \frac{1}{\delta} \leq
N \in \mathbb{N}.
\end{equation}
Taking $\phi_l (x) = \frac{1}{l^d}\phi(x/l)$ as in Lemma \ref{lem_stima}, we define:
$$v_0=v\ast \phi_l,\quad w_0=w\ast \phi_l, \quad A_0=A\ast \phi_l,
\quad {\mathcal{D}}_0= A_0-\big(\frac{1}{2}(\nabla v_0)^T\nabla v_0 + \sym\nabla w_0\big).$$
and observe the initial bounds \ref{pr_stima1}-\ref{pr_stima4} exactly as in the proof
of Theorem \ref{thm_stage} in section \ref{sec_stage}.

\bigskip

{\bf 2. (Induction definition: improving the deficit decomposition)}  Let
$\{\eta_i\in\mathbb{S}^{d-1}\}_{i=1}^{d_*}$, $r_0>0$ and the linear maps ${\bar a}_i$ be as in Lemma
\ref{lem_dec_def}. For $r=0\ldots N$ we iteratively define the perturbation amplitude vectors
$a^r = [a_i^r]_{i=1}^{d_*}\in\mathcal{C}^\infty(\bar\omega,\R^{d_*})$, by setting:
\begin{equation*}
\begin{split}
& a_i^0(x)=0 \quad \mbox{ for all }\; i=1\ldots d_*, \;\; x\in\bar\omega,
\\ & a^r_i(x)= \Big(2\bar a_i\big(\tilde C\Id_d + \mathcal{D}_0(x)-\mathcal{E}_{r-1}(x)\big)\Big)^{1/2} 
\quad \mbox{ for all }\; i=1\ldots d_*, \; \; r = 1\ldots N,\;\; x\in\bar\omega,
\\ &  \mbox{with } ~\tilde C= \frac{2}{r_0}\Big(\|\mathcal{D}\|_0 + \|\mathcal{D}_0\|_0\Big),
\end{split}
\end{equation*}
where the error fields $\mathcal{E}_r \in\mathcal{C}^\infty(\bar\omega,\R^{d\times d}_\sym)$ are given
by the right hand side of (\ref{step_errK}):
\begin{equation*}
\begin{split}
& \mathcal{E}_{r}=-\frac{1}{\lambda} \sum_{i=1}^{d_*}a_i^r \Big(G(\lambda t_{\eta_i})\nabla^2 v_0^i +
\bar G(\lambda t_{\eta_i})\nabla^2 v_0^{d_*+i}\Big) +
\frac{1}{2\lambda^2} \sum_{i=1}^{d_*} \nabla a_i^r\otimes \nabla a_i^r
\quad \mbox{ for all }\; r = 0\ldots N.
\end{split}
\end{equation*}
Our definition of $a^r$ is correctly posed if only $\Id_d + \frac{1}{\tilde C}\big
(\mathcal{D}_0 (x)-\mathcal{E}_{r-1}(x)\big)\in
B(\Id_d,r_0)\subset\R^{d\times d}_\sym$ for all $x\in \bar\omega$. 
To this end, we will prove that $\lambda l$ large enough (in
function of $\omega$ and $N$) implies:
\begin{equation}\label{need_errK}
\|\mathcal{E}_r\|_0\leq \frac{r_0\tilde C }{2} \quad \mbox{ for all } \; r=0\ldots N-1.
\end{equation}
Note that then automatically there holds for all $r=1\ldots N$:
\begin{equation}\label{low_bd_asK}
\begin{split}
& \tilde C\Id_d + \mathcal{D}_0-\mathcal{E}_{r-1} = \frac{1}{2}\sum_{i=1}^{d_*}
(a_i^r)^2\eta_i\otimes\eta_i \\ & \mbox{and }\quad (a_i^r)^2\geq
2r_0\tilde C \;\mbox{ in } \;\bar\omega,\;\mbox{ for all } \; i=1\ldots d_*.
\end{split}
\end{equation}

We now develop some preliminary estimates. Firstly, by  the linearity of $\bar
a_i$ in Lemma \ref{lem_dec_def}:
\begin{equation}\label{as0K}
\|a^r\|_0\leq C\| \tilde C \Id_d +
\mathcal{D}_0-\mathcal{E}_{r-1}\|_0^{1/2}\leq C \big(\|\mathcal{D}\|_0
+ \|\mathcal{E}_{r-1}\|_0\big)^{1/2}.
\end{equation}
Secondly, by the  Fa\'a di Bruno formula, exactly as proved in
(\ref{asm}), there holds for $m\geq 1$: 
\begin{equation}\label{asmK}
\|\nabla^{(m)}a^r\|_0 \leq C \|a^r\|_0\sum_{p_1+2p_2+\ldots mp_m=m}\prod_{t=1}^m
\Big(\frac{\|\nabla^{(t)}(\mathcal{D}_0-\mathcal{E}_{r-1})\|_0}{\tilde C}\Big)^{p_t}.
\end{equation}
Thirdly, applying Fa\'a di Bruno's formula to the inverse rather
than the square root, we get:
\begin{equation}\label{asm2K}
\begin{split}
& \|\nabla^{(m)}\Big(\frac{1}{a^{r}_i+a^{r-1}_i}\Big)\|_0 \\ & \leq \frac{C}{\tilde
C^{1/2}} \sum_{p_1+2p_2+\ldots mp_m=m}\prod_{t=1}^m
\Big(\frac{\|\nabla^{(t)}(a^{r} + a^{r-1})\|_0}{\tilde
  C^{1/2}}\Big)^{p_t} \quad\mbox{ for all }\; i=1\ldots d_*.
\end{split}
\end{equation}
The formulas (\ref{as0K}), (\ref{asmK}), (\ref{asm2K}) hold for all
$r=1\ldots N$ with constants $C$ depending on $\omega$, $m$.

\bigskip

{\bf 3. (Inductive estimates)} In the next step we will prove the following estimates:
\begin{align*}
&  \|a^r\|_0\leq C\|\mathcal{D}\|_0^{1/2} \quad\mbox{
  for all }\; r=1\ldots N, \tag*{(\theequation)$_1$}\refstepcounter{equation} \label{Fbound1K}\medskip\\
& \|\nabla^{(m)} a^r\|_0\leq C \frac{\lambda^{m}}{\lambda l}\|\mathcal{D}\|_0^{1/2} 
\quad\mbox{ for all }\; r=1\ldots N, \; \; m\geq 1, \tag*{(\theequation)$_2$}\label{Fbound2K} \medskip\\
& \|\nabla^{(m)} \big(\mathcal{E}_r-\mathcal{E}_{r-1}\big)\|_0 \leq
C \frac{\lambda^{m}}{(\lambda l)^r}\|\mathcal{D}\|_0
\quad\mbox{ for all }\; r=1\ldots N, \;\;
m\geq 0, \tag*{(\theequation)$_3$}\label{Fbound3K}
\end{align*}
with constants $C$ that depend only on $\omega$, $r$ and $m$. In general,
$C\to\infty$ as $m\to \infty$ or $r\to \infty$, so it is crucial that eventually
only finitely many of bounds above are used.
In particular, we note that \ref{Fbound3K} implies (\ref{need_errK}) provided
that $\lambda l$ surpasses the sum of constants $C$ corresponding to
$m=0$ and $r=1\ldots N$. This is achieved by taking $\sigma = (\lambda
l)^N\geq \sigma_0$ where the constant $\sigma_0\sim C^{1/\delta}\gg 1$
depends only on $\omega$ and $\delta$. 

\medskip

\noindent We now check that \ref{Fbound1K}-\ref{Fbound3K} are already valid at their
lowest counter value $r=1$. Indeed, \ref{Fbound1K} is a consequence of 
(\ref{as0K}), while \ref{Fbound2K} further follows from (\ref{asmK}) in view of \ref{pr_stima4}:
\begin{equation*}
\begin{split}
& \|\nabla^{(m)}a^1\|_0  \leq C \|a^1\|_0\sum_{p_1+2p_2+\ldots mp_m=m}\prod_{t=1}^m
\Big(\frac{\|\nabla^{(t)}\mathcal{D}_0\|_0}{\tilde C}\Big)^{p_t}
\\ & \leq C \|\mathcal{D}\|_0^{1/2} \sum_{p_1+2p_2+\ldots mp_m=m}\prod_{t=1}^m
\Big(\frac{\|\mathcal{D}\|_0}{\tilde C l^t}\Big)^{p_t} \leq 
\frac{C}{l^m}\|\mathcal{D}\|_{0}^{1/2}= C\frac{\lambda^m}{(\lambda
  l)^m}\|\mathcal{D}\|_{0}^{1/2}\leq C\frac{\lambda^m}{\lambda
  l}\|\mathcal{D}\|_{0}^{1/2}.
\end{split}
\end{equation*}
For the estimate
\ref{Fbound3K}, with the help of the above we compute at $m=0$:
\begin{equation*}
\|\mathcal{E}_1\|_0\leq C\Big(\frac{\|a^1\|_0\|\nabla^2v_0\|_0}{\lambda} +
\frac{\|\nabla a^1\|^2_0}{\lambda^2}\Big) \leq
C\Big(\frac{\|\mathcal{D}\|_0}{\lambda l} +
\frac{\|\mathcal{D}\|_0}{(\lambda l)^2} \Big) \leq C\frac{\|\mathcal{D}\|_0}{\lambda l},
\end{equation*}
and further, for all $m\geq 1$ in view of
\ref{pr_stima3} and with $C$ depending on $\omega$ and $m$:
\begin{equation*}
\begin{split}
& \|\nabla^{(m)}\mathcal{E}_1\|_0\leq
C\sum_{p+q+t=m}\lambda^{p-1}\|\nabla^{(q)} a^1\|_0\|\nabla^{(t+2)}v_0\|_0 + 
C\sum_{q+t=m}\lambda^{-2}\|\nabla^{(q+1)} a^1\|_0\|\nabla^{(t+1)}a^1\|_0 
\\ & \leq C\Big(\sum_{p+q=m}\frac{\lambda^{p-1}}{l^{t+1}}\|\mathcal{D}\|_0
+ \sum_{p+q+t=m, q\neq 0} \frac{\lambda^{p+q-1}}{(\lambda l) l^{t+1}}\|\mathcal{D}\|_0
+ \sum_{q+t=m} \frac{\lambda^{q+t+2}}{\lambda^2(\lambda l)^2}\|\mathcal{D}\|_0\Big)
\leq  C \frac{\lambda^{m}}{\lambda l}\|\mathcal{D}\|_0.
\end{split}
\end{equation*}

\bigskip

{\bf 4. (Proof of the inductive estimates)} Assume
that \ref{Fbound1K}--\ref{Fbound3K} hold, up to some counter value
$1\leq r\leq N-1$ and all $m\geq 0$. We will prove their validity at
$r+1$. By (\ref{as0K}) and \ref{Fbound3K} we directly get \ref{Fbound1K}:
$$\|a^{r+1}\|_0\leq C\Big(\|\mathcal{D}\|_0 +
\sum_{j=1}^{r}\|\mathcal{E}_j - \mathcal{E}_{j-1}\|_0\Big)^{1/2}\leq C
\big( 1+ \frac{1}{\lambda l}\big)^{1/2} \|\mathcal{D}\|_0^{1/2}\leq 
C \|\mathcal{D}\|_0^{1/2}.$$
Similarly, from (\ref{asmK}) and \ref{pr_stima4} we conclude \ref{Fbound2K} with $C$
depending on $\omega$ and $m$:
\begin{equation*}
\begin{split}
\|\nabla^{(m)}a^{r+1}\|_0 & \leq C \|a^{r+1}\|_0\sum_{p_1+2p_2+\ldots mp_m=m}\prod_{t=1}^m
\Big(\frac{\|\nabla^{(t)}\mathcal{D}_0\|_0 +
  \sum_{j=1}^r\|\nabla^{(t)}(\mathcal{E}_{j} - \mathcal{E}_{j-1})\|_0}{\tilde C}\Big)^{p_t}
\\ & \leq C \|\mathcal{D}\|_0^{1/2} \sum_{p_1+2p_2+\ldots mp_m=m}\prod_{t=1}^m
\Big(\frac{1}{l^t}+\frac{\lambda^t}{\lambda l}\Big)^{p_t} \\ & \leq 
C \|\mathcal{D}\|_0^{1/2} \sum_{p_1+2p_2+\ldots mp_m=m}
\frac{\lambda^m}{(\lambda l)^{p_1+p_2+\ldots p_m}}\leq 
C\frac{\lambda^{m}}{\lambda l}\|\mathcal{D}\|_0^{1/2}.
\end{split}
\end{equation*}
Towards showing \ref{Fbound3K}, we first deduce from the identity in (\ref{low_bd_asK}) that:
$$\mathcal{E}_{r}-\mathcal{E}_{r-1} = -\frac{1}{2}\sum_{i=1}^{d_*}
\big((a_i^{r+1})^2-(a_i^{r})^2\big) \eta_i\otimes \eta_i,$$
and hence for all $m\geq 0$ we get:
\begin{equation}\label{muK}
\|\nabla^{(m)} \big((a_i^{r+1})^2-(a_i^{r})^2\big)\|_0\leq
C \|\nabla^{(m)}\big(\mathcal{E}_{r}-\mathcal{E}_{r-1}\big) \|_0\leq
C\frac{\lambda^{m}}{(\lambda l)^{r}}\|\mathcal{D}\|_0 \quad\mbox{
  for all }\; i=1\ldots d_*.
\end{equation}
This in particular implies that by recalling the lower bound in (\ref{low_bd_asK}):
\begin{equation}\label{mu1K}
\|{a^{r+1}_i-a^{r}_i}\|_0 \leq \frac{C}{\tilde C^{1/2}}\|(a_i^{r+1})^2-(a_i^{r})^2\|_0
\leq \frac{C}{(\lambda l)^{r}}\|\mathcal{D}\|_0^{1/2} \quad\mbox{
  for all }\; i=1\ldots d_*.
\end{equation}
To estimate derivatives of $(a^{r+1}-a^{r})$, we recall (\ref{asm2K})
and observe that for every $m\geq 1$:
\begin{equation*}
\begin{split}
& \|\nabla^{(m)}\Big(\frac{1}{a^{r+1}_i+a^{r}_i}\Big)\|_0 \leq \frac{C}{\tilde
C^{1/2}} \sum_{p_1+2p_2+\ldots mp_m=m}\prod_{t=1}^m
\Big(\frac{\lambda^t}{\lambda l}\Big)^{p_t} \leq \frac{C}{\tilde C^{1/2}} 
\frac{\lambda^m}{\lambda l} \quad\mbox{ for all }\; i=1\ldots d_*,
\end{split}
\end{equation*}
which in combination with (\ref{muK}) yields for all $m\geq 1$:
\begin{equation}\label{mu2K}
\begin{split}
& \|\nabla^{(m)}\big(a_i^{r+1}-a_i^{r}\big)\|_0\leq C \sum_{q+t=m}
\|\nabla^{(q)}\big((a_i^{r+1})^2-(a_i^{r})^2\big)\|_0 
\|\nabla^{(t)}\Big(\frac{1}{a^{r+1}_i+a^{r}_i}\Big)\|_0 \\ & \leq 
C \frac{\lambda^{m}}{(\lambda l)^{r}}\frac{\|\mathcal{D}\|_0}{\tilde C^{1/2}} 
+ \sum_{q+t=m, t\neq 0} C \frac{\lambda^{q}}{(\lambda
  l)^{r}}\frac{\|\mathcal{D}\|_0}{\tilde C^{1/2}} \frac{\lambda^t}{\lambda l}
\\ & \leq C\frac{\lambda^m}{(\lambda l)^{r}}\|\mathcal{D}\|_0^{1/2} 
\quad\mbox{ for all }\; i=1\ldots d_*,
\end{split}
\end{equation}

We are now ready to estimate the derivatives of:
\begin{equation*}
\begin{split}
\mathcal{E}_{r+1}-\mathcal{E}_r = & -\frac{1}{\lambda}
\sum_{i=1}^{d_*}\big(a_i^{r+1}-a_i^{r}\big) \Big(\Gamma(\lambda t_{\eta_i})\nabla^2 v_0^i + 
\bar\Gamma(\lambda t_{\eta_i})\nabla^2 v_0^{d_*+i}\Big)  \\ & +
\frac{1}{2\lambda^2} \sum_{i=1}^{d_*} \Big(\big(\nabla a_i^{r+1} - \nabla a_i^r\big)\otimes \nabla a_i^{r+1} + 
\nabla a_i^r \otimes \big(\nabla a_i^{r+1} - \nabla a_i^r\big)\Big).
\end{split}
\end{equation*}
Namely, we get for all $m\geq 0$:
\begin{equation*}
\begin{split}
& \|\nabla^{(m)}(\mathcal{E}_{r+1}-\mathcal{E}_r)\|_0 \leq  C 
\sum_{p+q+t=m}\lambda^{p-1}
\|\nabla^{(q)}\big(a^{r+1}-a^{r}\big)\|_0 \|\nabla^{(t+2)}v_0\|_0
\\ & \qquad\qquad\qquad\qquad \quad
+ C \sum_{q+t=m} \lambda^{-2} \|\nabla^{(q+1)} (a^{r+1} - a^r)\|_0
\big(\|\nabla^{(t+1)} a^{r+1} \|_0+ \|\nabla^{(t+1)} a^{r} \|_0\big)
\\ & \leq \sum_{p+q+t=m}\lambda^{p-1} \frac{\lambda^q}{(\lambda l)^{r}l^{t+1}}\|\mathcal{D}\|_0 
+ C \sum_{q+t=m} \lambda^{-2} \frac{\lambda^{q+1}}{(\lambda l)^{r}}
\frac{\lambda^{t+1}}{\lambda l} \|\mathcal{D}\|_0 \leq C
\frac{\lambda^{m}}{(\lambda l)^{r+1}} \|\mathcal{D}\|_0.
\end{split}
\end{equation*}
This ends the proof of the last inductive estimate \ref{Fbound3K}. Observe that the closing of the
bounds as above, was possible due to the absence of the error term
$\frac{1}{\lambda^2} a\bar\Gamma(\lambda t_\eta)\nabla^2a$ in the right
hand side of (\ref{step_errK}), in the step construction generated by Nash's spirals. This term, appearing in
(\ref{step_err}) as the second order deficit generated by Kuiper's corrugations, would result in 
$\mathcal{E}_{r+1}-\mathcal{E}_r$ containing expressions of the form 
$\frac{1}{\lambda^2} a_i^r\bar\Gamma(\lambda t_{\eta_i})(\nabla^2
a^{r+1}-\nabla^2 a^r)$ which do not allow for the gain in the power
of $(\lambda l)$, because each $\|a^r\|_0$ is only of order $1$ in $\tilde
C$, see the bound in \ref{Fbound1K}.

\bigskip

{\bf 5. (End of proof)} 
Define $\tilde v\in\mathcal{C}^\infty(\bar\omega,\R^k)$ and
$\tilde w\in\mathcal{C}^\infty(\bar\omega, \R^d)$ according to the ``step''
construction in Lemma \ref{lem_stepK}, involving the periodic functions
$G, \bar G$ and the notation $t_\eta=\langle x, \eta\rangle$:
\begin{equation}\label{vw_finK}
\begin{split}
& \tilde v= v_1, \qquad \qquad\qquad v_1(x)= v_0(x) + 
\frac{1}{\lambda} \sum_{i=1}^{d_*} a_i^N(x)\big(G(\lambda t_{\eta_i})e_i 
+ \bar G(\lambda t_{\eta_i})e_{d_*+i}\big),\\
& \tilde w = w_1-\tilde C id_d, \qquad
w_1(x) = w_{0}(x) - \frac{1}{\lambda} \sum_{i=1}^{d_*}
a_i^N(x)\big(G(\lambda t_{\eta_i})\nabla v^i_0 
+ \bar G(\lambda t_{\eta_i})\nabla v_0^{d_*+i}\big).
\end{split}
\end{equation}
We now show that \ref{Fbound1K} - \ref{Fbound3K} imply the bounds
claimed in the Theorem. To prove \ref{Abound1K}, we use
\ref{Fbound1K}, \ref{Fbound2K} and \ref{pr_stima1}, \ref{pr_stima3}:
\begin{equation*}
\begin{split}
& \|\tilde v - v\|_1\leq \|v_0-v\|_1 + C\Big(\|a^N\|_0+\frac{\|\nabla a^N\|_0}{\lambda}\Big)\leq
C \|\mathcal{D}\|_0^{1/2} \big(1+\frac{1}{\lambda l}\big) \leq C \|\mathcal{D}\|_0^{1/2} ,\\
& \|\tilde w - w\|_1\leq \|w_0-w\|_1 + 
C\Big(\tilde C + \|a^N\|_0\|\nabla v_0\|_0 + \frac{\|\nabla a^N\|_0\|\nabla
  v_0\|_0 + \|a^N\|_0\|\nabla^2 v_0\|_0}{\lambda} \Big) \\
& \qquad \qquad \leq C \|\mathcal{D}\|_0^{1/2} + C \|\mathcal{D}\|_0^{1/2}\Big(\|\mathcal{D}\|_0^{1/2} +
(\|\mathcal{D}\|_0^{1/2}+\|\nabla v\|_0) + 
\frac{(\|\mathcal{D}\|_0^{1/2}+\|\nabla v\|_0) + \|\mathcal{D}\|_0^{1/2}}{\lambda l} \Big)\\
& \qquad \qquad \leq C \|\mathcal{D}\|_0^{1/2} \big(1+ \|\nabla v\|_0\big).
\end{split}
\end{equation*}
Similarly, there follows \ref{Abound2K} when we recall that $\lambda l
= \sigma^{1/N}\leq\sigma^\delta$ from (\ref{defi_lK}):
\begin{equation*}
\begin{split}
& \|\nabla^2\tilde v\|_0\leq \|\nabla^2 v_0\|_0 +
C\Big(\lambda\|a^N\|_0 + \|\nabla a^N\|_0 +
\frac{\|\nabla^2a^N\|_0}{\lambda}\Big) \\ & \qquad\quad \leq
C \|\mathcal{D}\|_0^{1/2}\Big(\frac{1}{l} + \lambda \Big) 
 = CM(1+ \lambda l) \leq CM\sigma^{1/N},\\
&\|\nabla^2\tilde w\|_0\leq \|\nabla^2 w_0\|_0 + 
C\Big(\lambda\|a^N\|_0\|\nabla v_0\|_0 + \big(\|\nabla a^N\|_0\|\nabla
  v_0\|_0 + \|a^N\|_0\|\nabla^2 v_0\|_0\big)
\\ & \qquad\qquad\qquad\qquad\qquad \qquad  + \frac{\|\nabla^2a^N\|_0\|\nabla v_0\|_0 + \|\nabla
  a^N\|_0\|\nabla^2 v_0\|_0 + \|a^N\|_0\|\nabla^3 v_0\|_0  }{\lambda}\Big)  \\ & 
\qquad\quad \leq C \|\mathcal{D}\|_0^{1/2}\Big(\frac{1}{l} +
\big(\lambda +\frac{1}{l}\big) (\|\mathcal{D}\|_0^{1/2} + \|\nabla v\|_0) + 
\frac{\|\mathcal{D}\|_0^{1/2}}{\lambda l^2}\Big) 
\\ & \qquad\quad \leq CM(1+ \lambda l) (1+\|\nabla v\|_0)
\leq CM\sigma^{1/N} (1+\|\nabla v\|_0).
\end{split}
\end{equation*}
Finally, (\ref{step_errK}) and (\ref{low_bd_asK}) yield \ref{Abound3K},
because we decompose:
$$\tilde{\mathcal{D}} = (A- A_0) + \mathcal{D}_0 - \Big(\frac{1}{2}
\sum_{i=1}^{d_*} (a_i^N)^2 \eta_i\otimes\eta_i + \mathcal{E}_N - \tilde C\Id_d\Big) 
= (A- A_0) - (\mathcal{E}_N - \mathcal{E}_{N-1}),$$
and further, in view of \ref{pr_stima2}, \ref{Fbound3K}:
\begin{equation*}
\begin{split}
\|\tilde{\mathcal{D}}\|_0 & \leq  \|A-A_0\|_0 + \|\mathcal{E}_N -
\mathcal{E}_{N-1}\|_0 \\ & \leq
C\Big(l^\beta\|A\|_{0,\beta} + \frac{\|\mathcal{D}\|_0}{(\lambda l)^N}\Big) = 
C\Big(\frac{\|A\|_{0,\beta}}{M^\beta}  \|\mathcal{D}\|_0^{\beta/2} +
\frac{\|\mathcal{D}\|_0}{\sigma}\Big).
\end{split}
\end{equation*}
The proof is done.
\endproof

\section{The Nash-Kuiper scheme in $\mathcal{C}^{1,\alpha}$}\label{sec4}

To perform induction on stages we need the following argument, 
similar to \cite[Theorem 1.1]{CDS}:

\begin{theorem}\label{th_NashKuiHol}
Let $\omega\subset\R^d$ be an open, bounded domain 
and let $k\geq 1$ and $\gamma>0$ be such that the statement of Theorem \ref{thm_stage}
holds true with $\gamma$ replacing the exponent $d_*/k$ in
\ref{Abound2}, provided that $\sigma>\sigma_0$ where $\sigma_0>1$
depends only on $\omega$ and $\gamma$. Then we have the following.  
For every $v\in\mathcal{C}^2(\bar\omega,\R^k)$, $w\in\mathcal{C}^2(\bar\omega,\R^d)$,
$A\in\mathcal{C}^{0,\beta}(\bar\omega, \R^{d\times d}_\sym)$, such that:
$$\mathcal{D}=A-\big(\frac{1}{2}(\nabla v)^T\nabla v + \sym\nabla
w\big) \quad\mbox{ satisfies } \quad 0<\|\mathcal{D}\|_0\leq 1,$$
and for every $\alpha$ in the range:
\begin{equation}\label{rangeAl}
0< \alpha <\min\Big\{\frac{\beta}{2},\frac{1}{1+2\gamma}\Big\},
\end{equation}
there exist $\tilde v\in\mathcal{C}^{1,\alpha}(\bar\omega,\R^k)$ and
$\tilde w\in\mathcal{C}^{1,\alpha}(\bar\omega,\R^d)$ with the following properties:
\begin{align*}
& \|\tilde v - v\|_1\leq C\|\mathcal{D}\|_0^{1/2}, \quad \|\tilde w -
w\|_1\leq C\|\mathcal{D}\|_0^{1/2}(1+\|\nabla v\|_0),
\tag*{(\theequation)$_1$}\refstepcounter{equation} \label{Hbound1}\vspace{1mm}\\
& A-\big(\frac{1}{2}(\nabla \tilde v)^T\nabla \tilde v + \sym\nabla
\tilde w\big) =0 \quad\mbox{ in }\; \bar\omega, \tag*{(\theequation)$_2$}\label{Hbound2} 
\end{align*}
where the constants $C$ depend only on $d, k$ and $\omega$.
\end{theorem}
\begin{proof}
{\bf 1.} Because of the assumption (\ref{rangeAl}), there exists an exponent:
\begin{equation}\label{exp_del}
\frac{2\gamma\alpha}{(1-\alpha)}< \delta < \min\Big\{1,\frac{2\gamma\beta}{(2-\beta)}\Big\}.
\end{equation} 
We let $\sigma >\sigma_0$ be a sufficiently large constant, in function of
$\delta,\alpha,\gamma$, $\|\nabla v\|_0$ and 
all constants $C$ in the assumed assertions of Theorem
\ref{thm_stage} (these constants depend only in $d,k, \omega$). 

We further set $v_0=v$, $w_0=w$, $\mathcal{D}_0=\mathcal{D}$, and take $M_0\geq \max\{\|v_0\|_2,
\|w_0\|_2,1\}$ that is again sufficiently large, now in function of
$\|A\|_{0,\beta}\|\mathcal{D}\|_0^{\beta/2-1} \sigma^\delta$ and
constants $C$ indicated before.
By successive applications of Theorem \ref{thm_stage} with the chosen $\sigma$ 
and constants $\{M_i\geq 1\}_{i=1}^\infty$ in:
$$M_i = \Big(\tilde C (1+\|\nabla v\|_0) \sigma^\gamma\Big)^iM_0$$
where $\tilde C>1$ is again some large constant (in function of
the aforementioned $C$),
we obtain sequences $\{v_i\in \mathcal{C}^2(\bar\omega,\R^k)\}_{i=1}^\infty$, 
$\{w_i\in\mathcal{C}^2(\bar\omega,\R^d)\}_{i=1}^\infty$ and the related
deficits $\{\mathcal{D}_i \in \mathcal{C}^0(\bar\omega,\R^{d\times d}_\sym)\}_{i=1}^\infty$:
$$\mathcal{D}_i = A-\big(\frac{1}{2}(\nabla v_i)^T\nabla v_i + \sym\nabla w_i\big).$$
We see that, as long as there holds:
\begin{equation}\label{indu_assuD}
0<\|\mathcal{D}_i\|_0\leq 1 \quad \mbox{ and }\quad M_i\geq \max\{\|v_i\|_2,\|w_i\|_2,1\},
\end{equation}
we have, with the constants $C$ depending only on $d, k$ and $\omega$:
\begin{align*}
& \|v_{i+1} - v_i\|_1\leq C\|\mathcal{D}_i\|_0^{1/2}, \quad \|w_{i+1} -
w\|_1\leq C\|\mathcal{D}_i\|_0^{1/2}(1+\|\nabla v_i\|_0),
\tag*{(\theequation)$_1$}\refstepcounter{equation} \label{Bbound1}\vspace{1mm}\\
& \|v_{i+1}\|_2\leq CM_i\sigma^\gamma,\quad \|w_{i+1}\|_2\leq CM_i\sigma^\gamma(1+\|\nabla v_i\|_0),
\tag*{(\theequation)$_2$}\label{Bbound2} \\ 
& \|{\mathcal{D}}_{i+1}\|_0\leq C\Big(\frac{\|A\|_{0,\beta}}{M_i^\beta} \|\mathcal{D}_i\|_0^{\beta/2} +
\frac{\|\mathcal{D}_i\|_0}{\sigma}\Big). \tag*{(\theequation)$_3$} \label{Bbound3}
\end{align*}
Below, we inductively validate  (\ref{indu_assuD}) for all $i\geq 0$,
and in fact we show that:
\begin{equation}\label{import_bd}
\|\mathcal{D}_i\|_0\leq \frac{1}{\sigma^{\delta i}}\|\mathcal{D}\|_0\quad
\mbox{ for all } \; i=0\ldots\infty.
\end{equation}

\smallskip

Before doing so, note that (\ref{import_bd}) actually implies both
statements in (\ref{indu_assuD}). Indeed, by \ref{Bbound2}:
\begin{equation*}
\begin{split}
& \|v_{i+1}\|_2\leq CM_i\sigma^\gamma\leq M_{i+1},\vspace{1mm}\\
& \|w_{i+1}\|_2\leq CM_i\sigma^\gamma(1+\|\nabla v_i\|_0) 
\leq CM_i\sigma^\gamma(1+2C +\|\nabla v\|_0) \leq M_{i+1},
\end{split}
\end{equation*}
since by the first bound in \ref{Bbound1} and (\ref{import_bd}) there
follows, provided that $\sigma^{\delta/2}\geq 2$:
\begin{equation}\label{pomoc_v}
\begin{split}
\|\nabla v_i\|_0 & \leq \|\nabla v\|_0 + \sum_{j=0}^{i-1}\|\nabla v_{j+1}-\nabla v_j\|_0
\leq \|\nabla v\|_0 + C\sum_{j=0}^{i-1}\|\mathcal{D}_j\|_0^{1/2} \\ & \leq 
\|\nabla v\|_0 + C\sum_{j=0}^{\infty}\frac{\|\mathcal{D}\|_0^{1/2}}{\sigma^{\delta j/2}} 
= \|\nabla v\|_0 + \frac{C}{1-\sigma^{-\delta/2}} \|\mathcal{D}\|_0^{1/2}
\\ & \leq \|\nabla v\|_0 + {2C} \|\mathcal{D}\|_0^{1/2} \leq 2C +\|\nabla v\|_0.
\end{split}
\end{equation}

\medskip

{\bf 2.} Clearly (\ref{import_bd}) holds at $i=0$. To prove
it at $(i+1)$, use \ref{Bbound3} and the induction assumption:
\begin{equation}\label{dd}
\|\mathcal{D}_{i+1}\|_0\leq C\Big( \frac{\|A\|_{0,\beta}
  \|\mathcal{D}\|_0^{\beta/2}}{M_i^\beta\sigma^{\delta i\beta/2}} +
\frac{\|\mathcal{D}\|_0}{\sigma^{\delta i +1}}\Big) = \frac{\|\mathcal{D}\|_{0}}{\sigma^{\delta(i+1)}}
\Big(\frac{C\|A\|_{0,\beta} \|\mathcal{D}\|_0^{\beta/2-1}}{M_i^\beta\sigma^{\delta i\beta/2 - \delta(i+1)}} +
\frac{C}{\sigma^{1-\delta}}\Big),
\end{equation}
and check that both terms in parentheses in the right hand side above
are not greater than $1/2$. For the second term, this is readily implied by
taking $\sigma$ large enough that $\sigma^{1-\delta}\geq 2C$, in view
of $1-\delta>0$ in (\ref{exp_del}). For the first term, we note that
$$\frac{C\|A\|_{0,\beta} \|\mathcal{D}\|_0^{\beta/2-1}}{M_i^\beta\sigma^{\delta i\beta/2 -  \delta(i+1)}} 
\leq \frac{C\|A\|_{0,\beta} \|\mathcal{D}\|_0^{\beta/2-1}\sigma^\delta}{M_0^\beta}
\cdot \sigma^{\delta i-\gamma\beta i - \delta\beta i/2} \leq \frac{C\|A\|_{0,\beta}
    \|\mathcal{D}\|_0^{\beta/2-1}\sigma^\delta}{M_0^\beta},$$  
since the exponent $\delta i-\gamma\beta i - \delta\beta i/2$ is
non-positive, due to $\delta < \frac{2\gamma\beta}{2-\beta}$ in (\ref{exp_del}):
$$\delta i-\gamma\beta i - \delta\beta i/2 =\frac{i}{2}
\big(\delta(2 - \beta) - 2\gamma\beta \big) \leq 0 \quad\mbox{ for all }\; i\geq 0.$$
In conclusion, the expression in parentheses in (\ref{dd}) is bounded
by $1$, provided $M_0$ has been chosen sufficiently large. This ends the proof of (\ref{import_bd}).

\medskip

{\bf 3.} From \ref{Bbound1}, (\ref{pomoc_v}) and (\ref{import_bd}), it follows that for all $i=0\ldots \infty$:
\begin{equation*}
\begin{split}
& \|v_{i+1} - v_i\|_1\leq \frac{C}{\sigma^{\delta
    i/2}}\|\mathcal{D}\|_0^{1/2}, \qquad \|w_{i+1} - w_i\|_1\leq \frac{C}{\sigma^{\delta
    i/2}}\|\mathcal{D}\|_0^{1/2} \big(1+ \|\nabla v\|_0\big),
\end{split}
\end{equation*}
Hence, both sequences $\{v_i\}_{i=1}^\infty$, 
$\{w_i\}_{i=1}^\infty$ are Cauchy in $\mathcal{C}^1(\bar\omega)$ and as
such they converge to the limit fields, respectively: 
$$\tilde v\in\mathcal{C}^1(\omega,\R^k), \qquad \tilde w\in\mathcal{C}^1(\omega,\R^d)$$ 
that satisfy \ref{Hbound1} and \ref{Hbound2}, in virtue of
$\|\mathcal{D}_i\|_0\to 0$ as $i\to\infty$. 

\smallskip

It remains to show that $\tilde v$ and $\tilde w$ are
$\mathcal{C}^{1,\alpha}$-regular. To this end, we use the estimate:
\begin{equation*}
\begin{split}
\|v_{i+1} - v_i\|_2 + \|w_{i+1} - w_i\|_2 & \leq {C}M_i{\sigma^\gamma}
\big(1+ \|\nabla v\|_0\big) \\ & \leq C \Big(\tilde C (1+\|\nabla
v\|_0)\sigma^\gamma\Big)^{i+1} M_0, 
\end{split}
\end{equation*}
resulting from \ref{Bbound2} and (\ref{pomoc_v}), in the
interpolation inequality $\|\cdot\|_{1,\alpha}\leq C \|\cdot\|_{1}^{1-\alpha}\|\cdot\|_{2}^\alpha$:
\begin{equation*}
\begin{split}
\|v_{i+1} -& v_i\|_{1,\alpha}  + \|w_{i+1} - w_i\|_{1,\alpha} \\ & \leq {C}
\|\mathcal{D}\|_0^{(1-\alpha)/2} \Big(\tilde C (1+\|\nabla
v\|_0)\Big)^{(i+1)\alpha+(1-\alpha)} M_0^\alpha \sigma^{\alpha\gamma(i+1) - \delta i (1-\alpha)/2}
\\ & = {C} \tilde C \cdot M_0^\alpha\|\mathcal{D}\|_0^{(1-\alpha)/2} \big(1+\|\nabla
v\|_0\big) \sigma^{\alpha\gamma} \cdot \Big(\frac{\tilde C^\alpha (1+\|\nabla v\|_0)^{\alpha}}{
\sigma^{\delta (1-\alpha)/2-\alpha\gamma}}\Big)^i.
\end{split}
\end{equation*}
Since the exponent $\delta (1-\alpha)/2-\alpha\gamma$ is positive, in
view of $\delta>\frac{2\gamma\alpha}{1-\alpha}$ in (\ref{exp_del}), we
see that both sequences $\{v_i\}_{i=1}^\infty$, $\{w_i\}_{i=1}^\infty$ are Cauchy in
$\mathcal{C}^{1,\alpha}(\bar\omega)$, provided that $\sigma$ is
sufficiently large to have:
$$\frac{\tilde C^\alpha (1+\|\nabla v\|_0)^{\alpha}}{
\sigma^{\delta (1-\alpha)/2-\alpha\gamma}} <1.$$
In conclusion, $\tilde v\in\mathcal{C}^{1,\alpha}(\omega,\R^k)$ and
$\tilde w\in\mathcal{C}^{1,\alpha}(\omega,\R^d)$ as claimed. The proof
is done.
\end{proof}

\bigskip

Taking now $\gamma= d_*/k$ or taking an arbitrary $\gamma<1$ as
guaranteed by Theorem \ref{thm_stage} and Theorem \ref{thm_stageK}, respectively,
Theorem \ref{th_NashKuiHol} implies:

\begin{corollary}\label{th_NKH}
For every $v\in\mathcal{C}^2(\bar\omega,\R^k)$, $w\in\mathcal{C}^2(\bar\omega,\R^d)$ and
$A\in\mathcal{C}^{0,\beta}(\bar\omega, \R^{d\times d}_\sym)$ defined
on an open, bounded domain $\omega\subset\R^d$, and such that:
$$\mathcal{D}=A-\big(\frac{1}{2}(\nabla v)^T\nabla v + \sym\nabla
w\big) \quad\mbox{ satisfies } \quad 0<\|\mathcal{D}\|_0\leq 1,$$
and for every exponent $\alpha$ in the range:
\begin{equation*}
0< \alpha <\min\Big\{\frac{\beta}{2},\frac{1}{1+2d_*/k}\Big\},
\quad\mbox{ or }\;
0< \alpha <\min\Big\{\frac{\beta}{2},1\Big\} \mbox{ when $k\geq d(d+1)$},
\end{equation*}
there exist $\tilde v\in\mathcal{C}^{1,\alpha}(\bar\omega,\R^k)$ and
$\tilde w\in\mathcal{C}^{1,\alpha}(\bar\omega,\R^d)$ with the following properties:
\begin{align*}
& \|\tilde v - v\|_1\leq C\|\mathcal{D}\|_0^{1/2}, \quad \|\tilde w -
w\|_1\leq C\|\mathcal{D}\|_0^{1/2}(1+\|\nabla v\|_0),
\tag*{(\theequation)$_1$}\refstepcounter{equation} \label{HHbound1}\vspace{1mm}\\
& A-\big(\frac{1}{2}(\nabla \tilde v)^T\nabla \tilde v + \sym\nabla
\tilde w\big) =0 \quad\mbox{ in }\; \bar\omega, \tag*{(\theequation)$_2$}\label{HHbound2} 
\end{align*}
where the constants $C$ depend only on $d, k$ and $\omega$.
\end{corollary}

\section{Proofs of Theorem \ref{th_final} and Theorem \ref{th_finalK}}\label{sec5}

The final auxiliary result that we need, is a combination of the
local decomposition into ``primitive metrics'' with a partition of unity -
type statement from \cite[Lemma 3.3]{Laszlo}:

\begin{lemma}\label{lem_met_deco}
Given the dimension $d\geq 1$, there exists a constant $N_0$ and
sequences of unit vectors $\{\eta_i\in\R^d\}_{i=1}^\infty$ 
and nonnegative functions
$\{\varphi_i\in\mathcal{C}^\infty_c(\R^{d\times d}_{\sym, >},\R)\}_{i=1}^\infty$, such that:
$$ A = \sum_{i=1}^{\infty} \varphi_i(A)^2 \eta_i\otimes\eta_i \quad
\mbox{ for all } \; A\in \R^{d\times d}_{\sym, >},$$
and such that:
\begin{itemize} 
\item[(i)] at most $N_0$ terms in the above sum are nonzero,
\item[(ii)] every compact set of matrices $K\subset  \R^{d\times d}_{\sym, >}$ 
induces a finite set of indices $J(K)\subset\mathbb{N}$, such that
$\varphi_i(A)=0$ for all $A\in K$ and all $i\not\in J(K)$.
\end{itemize}
\end{lemma}

\medskip

\noindent Equipped with Lemma \ref{lem_met_deco} and the ``step''
construction in Lemma \ref{lem_step}, one easily deduces the
deficit decrease - approximation result in $\mathcal{C}^1$, which is
the multidimensional version of the basic ``stage'' construction in \cite[Proposition 3.2]{lewpak_MA}:

\begin{theorem}\label{th_approx_nonlocal}
Let $\omega\subset\R^d$ be an open, bounded domain. Given
two vector fields $v\in\mathcal{C}^\infty(\bar\omega,\R^k)$, $w\in\mathcal{C}^\infty(\bar\omega,\R^d)$ and
a matrix field $A\in\mathcal{C}^\infty(\bar\omega,\R^{d\times d}_\sym)$, assume that:
$$\mathcal{D}=A-\big(\frac{1}{2}(\nabla v)^T\nabla v + \sym\nabla
w\big) \quad\mbox{ satisfies } \quad \mathcal{D}>c\,\Id_d \; \mbox{ on }
\; \bar\omega$$
for some $c>0$, in the sense of matrix inequalities. Fix
$\epsilon>0$. Then, there exists $\tilde v\in\mathcal{C}^\infty(\bar \omega,\R^k)$ and
$\tilde w\in\mathcal{C}^\infty(\bar\omega,\R^d)$ such that, denoting:
$$\tilde{\mathcal{D}}=A -\big(\frac{1}{2}(\nabla \tilde v)^T\nabla \tilde v + \sym\nabla \tilde w\big), $$
the following holds with constants $C$ depending only on $d, k$ and $\omega$:
\begin{align*}
& \|\tilde v - v\|_0\leq \epsilon, \quad \|\tilde w - w\|_0\leq \epsilon,
\tag*{(\theequation)$_1$}\refstepcounter{equation} \label{Cbound1}\vspace{1mm}\\ 
& \|\nabla (\tilde v-v)\|_0\leq C \|\mathcal{D}\|_0^{1/2}, \quad \|\nabla(\tilde w -
w)\|_0\leq C\|\mathcal{D}\|_0^{1/2}\big(\|\mathcal{D}\|_0^{1/2} +\|\nabla v\|_0\big),
\tag*{(\theequation)$_2$} \label{Cbound2}\vspace{1mm}\\
& \|\tilde{\mathcal{D}}\|_0\leq \epsilon\quad\mbox{ and } \quad
\tilde{\mathcal{D}}>\tilde c\,\Id_d \; \mbox{ for some $\tilde c>0$}. \tag*{(\theequation)$_3$} \label{Cbound3}
\end{align*}
\end{theorem}
\begin{proof}
{\bf 1.} Since $\mathcal{D}(\bar\omega)$ is a compact subset of $\R^{d\times d}_{\sym,>}$, 
Lemma \ref{lem_met_deco} yields a finite set of indices for which the decomposition into ``primitive matrices'' is
active. Without loss of generality these indices are
$\{1\ldots N\}$. Then, with the unit vectors $\{\eta_i\in\R^d\}_{i=1}^N$ and the nonnegative functions
$\{b_i\in\mathcal{C}^\infty(\bar\omega,\R)\}_{i=1}^N$ defined by $b_i(x)=\varphi_i(\mathcal{D}(x))$, there holds: 
$$\mathcal{D}(x) = \sum_{i=1}^N b_i(x)^2 \eta_i\otimes\eta_i \qquad
\mbox{ for all }\; x\in\bar\omega.$$
We now define the modified nonnegative amplitude functions $\{a_i\in\mathcal{C}^\infty(\bar\omega,\R)\}_{i=1}^N$ by:
$$a_i = (1-\delta)^{1/2}b_i \quad\mbox{on }\;\omega, \quad\mbox{ where
}\; \delta = \min\big\{\frac{1}{2}, \frac{\epsilon}{2\|\mathcal{D}\|_0}\big\}.$$
Observe that:
$$\mathcal{D}-\sum_{i=1}^Na_i^2 \eta_i\otimes\eta_i = \delta\mathcal{D} >\delta c\,\Id_d.$$

\smallskip

{\bf 2.} We set $v_1 = v$ and $w_1=w$. We then inductively define the
vector fields $\{v_i\in\mathcal{C}^\infty(\bar \omega,\R^k)\}_{i=1}^{N+1}$ and
$\{w_i\in\mathcal{C}^\infty(\bar\omega,\R^d)\}_{i=1}^{N+1}$ by applying
Lemma \ref{lem_step} to each consecutive pair $(v_i, w_i)$ with the
given unit vector $\eta_i$, an arbitrary unit vector $E\in\R^k$, the
given amplitude $a_i$ and a frequency $\lambda_i>0$ that is sufficiently large as
indicated below. We finally set:
$$\tilde v = v_{N+1},\quad\tilde w = w_{N+1}.$$
It is clear that by taking $\{\lambda_i\}_{i=1}^N$ 
large, one can ensure the validity of \ref{Cbound1}. Further, by (\ref{step_err}):
\begin{equation*}
\begin{split}
\tilde{\mathcal{D}}& = \mathcal{D} - \Big(\big(\frac{1}{2}(\nabla
v_{N+1})^T\nabla v_{N+1} + \sym\nabla w_{N+1}\big) -
\big(\frac{1}{2}(\nabla v_1)^T\nabla v_1 + \sym\nabla w_1\big)\Big)\\  
& = \Big(\mathcal{D} - \sum_{i=1}^Na_i^2\eta_i\otimes\eta_i\Big) \\ &
\qquad - \sum_{i=1}^{N}\Big( \big(\frac{1}{2}(\nabla v_{i+1})^T\nabla v_{i+1} +
\sym\nabla w_{i+1}\big) - \big(\frac{1}{2}(\nabla v_i)^T\nabla v_i + \sym\nabla w_i\big)
-a_i^2\eta_i\otimes\eta_i\Big) \\ & = \delta\mathcal{D} + \sum_{i=1}^{N}
O\Big(\frac{\|a_i\|_0\|\nabla^2v_i\|_0}{\lambda_i} + \frac{\|\nabla
  a_i\|^2 +\|a_i\|_0\|\nabla^2a\|_0}{\lambda_i^2}\Big),
\end{split}
\end{equation*}
which implies \ref{Cbound3} with $\tilde c = \delta c/2$, provided
that $\{\lambda_i\}_{i=1}^N$ are sufficiently large.

\medskip

{\bf 3.} It remains to check \ref{Cbound2}. By Lemma
\ref{lem_met_deco} (i), we get for each $x\in\bar\omega$:
$$0\leq \sum_{i=1}^Na_i(x) \leq \sum_{i=1}^Nb_i(x) \leq N_0^{1/2}
\big(\sum_{i=1}^Nb_i(x)^2\big)^{1/2}  =   N_0^{1/2}
\big(\mathrm{Trace}\,\mathcal{D}(x)\big)^{1/2}
\leq C N_0^{1/2} \|\mathcal{D}\|_0^{1/2}.$$
Consequently, the definition (\ref{defi_per}) yields, with
sufficiently large $\{\lambda_i\}_{i=1}^N$: 
\begin{equation*}
\|\nabla (\tilde v-v)\|_0\leq \sum_{i=1}^N \|\nabla v_{i+1}-\nabla v_i\|_0 \leq
2 \|\sum_{i=1}^N  a_i\|_0 + C\sum_{i=1}^N \frac{\|\nabla a_i\|_0}{\lambda_i} 
\leq C \|\mathcal{D}\|_0^{1/2}.
\end{equation*}
In a similar fashion, and using the above bound, we obtain:
\begin{equation*}
\begin{split}
\|\nabla (\tilde w-w)\|_0& \leq \sum_{i=1}^N \|\nabla w_{i+1}-\nabla w_i\|_0 \leq
C\Big \|\sum_{i=1}^N  \big(\|\nabla v_i\|_0 +\|a_i\|_0\big) a_i\Big\|_0
\\ & \quad + C\sum_{i=1}^N \frac{\|\nabla a_i\|_0\|\nabla v_i\|_0 +
  \|a_i\|_0\|\nabla^2v_i\|_0 + \|a_i\|_0^2+\|a_i\|_0\|\nabla a_i\|_0}{\lambda_i}  
\\ & \quad + C \sum_{i=1}^N \frac{\|\nabla a_i\|_0\|\nabla^2 a_i\|_0 +
  \|a_i\|_0\|\nabla^3a_i\|_0 }{\lambda_i^2}  \\ & \leq C \|\mathcal{D}\|_0^{1/2}
\cdot \sup_{i=1\ldots N}\big(\|\nabla v_i\|_0+ \|a_i\|_0\big)
\leq C \|\mathcal{D}\|_0^{1/2} \big(\|\nabla v_0\|_0+ \|\mathcal{D}\|_0^{1/2}\big).
\end{split}
\end{equation*}
This ends the proof of \ref{Cbound2} and of the theorem.
\end{proof}

\smallskip

\noindent We remark that having the upgraded version of the ``step''
in Lemma \ref{lem_step} was irrelevant to the proof
above, and that the sub-optimal construction in \cite[Lemma
2.2]{lewpak_MA} would still suffice. 

\bigskip

\noindent We are now ready to give: 

\medskip

\noindent {\bf Proofs of Theorem \ref{th_final} and Theorem \ref{th_finalK}}

In order to apply Corollary \ref{th_NKH}, we need to increase the
regularity of $v$, $w$ and decrease  the deficit $\mathcal{D}$. Take
$\epsilon<1$ that is sufficiently small, as indicated below.
First, we let $v_1\in\mathcal{C}^\infty(\bar\omega,\R^k)$, $w_1\in\mathcal{C}^\infty(\bar\omega,\R^d)$ and 
$A_1\in \mathcal{C}^\infty(\bar\omega,\R^{d\times d}_{\sym})$ be such that:
\begin{equation*}
\begin{split}
&\|v_1- v\|_1\leq \epsilon^3,\quad \|w_1- w\|_1\leq \epsilon^3,\quad
\|A_1- A\|_0\leq \epsilon^3,\\
&\mathcal{D}_1= A_1 -\big(\frac{1}{2}(\nabla v_1)^T\nabla v_1 + \sym\nabla w_1\big)
> c_1\Id_d \quad\mbox{ for some }\; c_1>0.
\end{split}
\end{equation*}
The last property follows from the fact that:
\begin{equation}\label{osta}
\|\mathcal{D}_1-\mathcal{D}\|_0\leq 3\epsilon^3 + \epsilon^3 \|\nabla v\|_0 
\end{equation}
Second, use Theorem \ref{th_approx_nonlocal} to get $v_2\in\mathcal{C}^\infty(\bar\omega,\R^k)$,
$w_2\in\mathcal{C}^\infty(\bar\omega,\R^d)$ such that:
\begin{equation*}
\begin{split}
&\|v_2- v_1\|_0\leq \epsilon^3,\quad \|w_2- w_1\|_0\leq \epsilon^3,\\
& \|\nabla (v_2-v_1)\|_0\leq C\|\mathcal{D}_1\|_0^{1/2}\leq
C\big(\|\mathcal{D}\|_0^{1/2} + \epsilon^{3/2} + \|\nabla v\|_0^{1/2}\big),\\
&\mathcal{D}_2= A_1 -\big(\frac{1}{2}(\nabla v_2)^T\nabla v_2 + \sym\nabla w_2\big)
\quad\mbox{ satisfies }\; \|\mathcal{D}_2\|_0\leq \epsilon^3,
\end{split}
\end{equation*}
where we applied (\ref{osta}) in the gradient increment bound of $v$.

\smallskip

Clearly, if the deficit $\mathcal{D}_3$, defined below:
$$\mathcal{D}_3 = A- \big(\frac{1}{2}(\nabla v_2)^T\nabla v_2 + \sym\nabla w_2\big) $$
is equivalently zero on $\bar\omega$, then we may simply
take $\tilde v= v_2$ and $\tilde w=w_2$ to satisfy the claim of the
theorem. Otherwise, we use Corollary \ref{th_NKH} to $v_2$, $w_2$ and $A$, since:
$$0<\|\mathcal{D}_3\|_0 \leq \|A-A_1\|_0 + \|\mathcal{D}_2\|\leq 2\epsilon^3\leq 1,$$
and consequently obtain $\tilde v\in\mathcal{C}^{1,\alpha}(\bar\omega,\R^k)$,
$\tilde w\in\mathcal{C}^{1,\alpha}(\bar\omega,\R^d)$ such that:
\begin{equation*}
\begin{split}
&\|\tilde v - v_2\|_0\leq C\epsilon^{3/2},\\ &
\|\tilde w- w_2\|_0\leq C\epsilon^{3/2}(1+ \|\nabla v_2\|_0) \leq 
C\epsilon^{3/2}(1+ \|\mathcal{D}\|_0^{1/2} + \|\nabla v\|_0), 
\\ & A -\big(\frac{1}{2}(\nabla \tilde v)^T\nabla \tilde v + \sym\nabla
\tilde w\big) = 0 \quad\mbox{ in }\; \bar\omega.
\end{split}
\end{equation*} 
It now suffices to observe that, taking $\epsilon$ sufficiently small (in function of
$\|\mathcal{D}\|_0^{1/2}$, $\|\nabla v\|_0$ and of constants $C$
that depend only on $k$, $d$ and $\omega$), we get:
\begin{equation*}
\begin{split}
&\|\tilde v - v\|_0\leq \|\tilde v - v_2\|_0 + \| v_2 - v_1\|_0 + \|v_1 - v\|_0\leq
C\epsilon^{3/2} \leq \epsilon,\\ &
\|\tilde w - w\|_0\leq \|\tilde w - w_2\|_0 + \| w_2 - w_1\|_0 + \|w_1 - w\|_0
\leq C\epsilon^{3/2}(1+ \|\mathcal{D}\|_0^{1/2} + \|\nabla v\|_0) \leq \epsilon.
\end{split}
\end{equation*} 
The proof is done.
\endproof

\section{The Monge-Amp\`ere system: proofs of Lemma \ref{th_equiv},
  Lemma \ref{th_solve} and Theorem \ref{th_CI_weakMA}}\label{sec6}

We remark that when $d=2$ then the formula (\ref{RiemC2}) rests in agreement with 
the expansion of the Gaussian curvature:
$\kappa(\Id_2 + \epsilon A) = -\frac{\epsilon}{2} \curl\,\curl A +
O(\epsilon^2)$, because we have:
\begin{equation*}
\mathfrak{C}^2(A)_{ij,st}  = \left\{\begin{array}{ll} \curl\,\curl A &
    \mbox{ if }\, (ij,st)\in\{(12,12), (21,21)\},\\
-\curl\,\curl A & \mbox{ if }\, (ij,st)\in\{(12,21), (21,12)\},\\
0 &  \mbox{ otherwise.}
\end{array}
\right.
\end{equation*}

\noindent We now give:

\smallskip

\noindent {\bf Proof of Lemma \ref{th_equiv}}

The implication (i)$\Rightarrow$(ii) follows by a direct inspection:
\begin{equation*}
\begin{split}
2 \mathfrak{C}^2(\sym\nabla w)_{ij,st} = \; & \partial_i\partial_s (\partial_jw^t + \partial_tw^j)
+ \partial_j\partial_t (\partial_iw^s + \partial_sw^i) \\ & - \partial_i\partial_t (\partial_jw^s + \partial_sw^j)
- \partial_j\partial_s (\partial_iw^t + \partial_tw^i) =0.
\end{split}
\end{equation*}
To prove that (ii)$\Rightarrow$(i), note that condition:
$$\mathfrak{C}^2(A)_{ij,st}= \partial_i\big(\partial_sA_{jt}-\partial_tA_{js}\big)
- \partial_j\big(\partial_sA_{it} - \partial_tA_{is}\big)=0$$
implies, for each fixed $s,t:1\ldots d$, that the vector field
$[\partial_{s}A_{jt}-\partial_tA_{js}]_{j=1\ldots d} $ must be a gradient
of some scalar field $-\phi_{st}$, where we used the
Poincar\'e Lemma on the contractible domain $\omega$. We thus write:
\begin{equation}\label{uno}
\partial_s A_{jt} - \partial_tA_{js}= -\partial_j\phi_{st}\quad\mbox{
  for all }\; j,s,t=1\ldots d. 
\end{equation}
Observe that $\nabla(\phi_{st} + \phi_{ts})=0$, so without loss of
generality, $\phi_{st}=-\phi_{ts}$ and $\phi_{ss}=0$. The following
matrix field is thus skew-symmetric:
\begin{equation}\label{due}
\phi = [\phi_{st}]_{s,t=1\ldots d}: \omega\to \so(d).
\end{equation}
Consequently, permuting the indices in (\ref{uno}), leads to:
$$\partial_s A_{jt} - \partial_tA_{js} +\partial_j\phi_{st} = 0\quad\mbox{
  and } \quad\partial_t A_{js} - \partial_jA_{st} +\partial_s\phi_{tj} = 0.$$
Summing the above two expressions, we get:
$$\partial_s A_{tj} - \partial_jA_{ts} + \partial_{s}\phi_{tj}
- \partial_j\phi_{ts} = 0 \quad\mbox{ for all }\; j,s,t=1\ldots d. $$
Hence for any $i=1\ldots d$, the $i$-th row $[A_{ij} + \phi_{ji}]_{j=1\ldots d}$ of the
matrix field $A+\phi$ is a gradient of some scalar field $w^i$ on $\omega$, where
we again used Poincar\'e's Lemma. Writing $w=[w^i]_{i=1\ldots d}:\omega\to\R^d$,
we obtain the claim by taking the symmetric parts of the resulting identity:
$$A+\phi = \nabla w,$$
and invoking the skew-symmetry in  (\ref{due}). The proof is done.
\endproof

\smallskip

\noindent Note that the operator $\mathfrak{C}^2$ can be interpreted as $\curl\,\curl$, in any dimension $d$.
To see this, recall that when $d=2,3$ then the coefficients of $\curl \,w$ for a vector
field $w=[w^i]_{i=1\ldots d}$, coincide with the coefficients
of the exterior derivative $ \mathrm{d}\alpha =
\sum_{i<j}\big(\partial_iw^j-\partial_jw^i\big)\mathrm{d}x_i\wedge
\mathrm{d}x_j$ of the 1-form $\alpha=\sum_{i=1}^dw^i \mathrm{d}x_i$.
Given a matrix field $A$ in any dimension $d$, we may still apply $\mathrm{d}$ row-wise:
$$\mathrm{d}A=\Big[\sum_{s<t} \big(\partial_sA_{it} - \partial_tA_{is}\big)\mathrm{d}x_s\wedge
\mathrm{d}x_t\Big]_{i=1\ldots d},$$ 
returning a vector of 2-forms, and then apply $\mathrm{d}$ to each vector of coefficients in $\mathrm{d}A$:
\begin{equation*}
\begin{split}
\mathrm{d}^2A & = \Big[\sum_{i<j} \big(\partial_i(\partial_sA_{jt} - \partial_tA_{js})
- \partial_j(\partial_sA_{it} - \partial_tA_{is})\big) \mathrm{d}x_i\wedge
\mathrm{d}x_j\Big]_{s<t:1\ldots d} \\ & 
= \Big[\sum_{i<j} \mathfrak{C}^2(A)_{ij,st} \mathrm{d}x_i\wedge \mathrm{d}x_j \Big]_{s<t:1\ldots d}.
\end{split}
\end{equation*}

\medskip

\noindent Next, we show the equivalent solvability conditions
determining the range of $\mathfrak{C}^2$:

\smallskip

\noindent {\bf Proof of Lemma \ref{th_solve}}

The implication (i)$\Rightarrow$(ii) follows by a direct
inspection. To prove (ii)$\Rightarrow$(i), observe first that for a
skew-symmetric matrix field $B:\omega\to\R^{d\times
  d}_{\mathrm{skew}}$ to be of the form: 
$B= (\nabla w)^T - \nabla w$ for some $w=[w^i]_{i=1\ldots d}:\omega\to\R^d$, the
sufficient and necessary condition is: 
\begin{equation}\label{tre}
\partial_iB_{jq} + \partial_jB_{qi}+\partial_qB_{ij} = 0\quad\mbox{ for all }\; i,j,q=1\ldots d.
\end{equation}
This claim follows by taking the exterior derivative of the 2-form in:
\begin{equation*}
\begin{split}
\mathrm{d}\big(\sum_{j,q=1\ldots d}B_{jq}\mathrm{d}x_j\wedge\mathrm{d}x_q\big)
& = \sum_{i,j,q=1\ldots d}\partial_i B_{jq}\mathrm{d}x_i\wedge \mathrm{d}x_j\wedge\mathrm{d}x_q
\\ & = 2\sum_{i<j<q:1\ldots d}(\partial_i B_{jq} + \partial_jB_{qi}
+ \partial_qB_{ij})\mathrm{d}x_i\wedge \mathrm{d}x_j\wedge\mathrm{d}x_q,
\end{split}
\end{equation*}
where we used the skew-symmetry assumption, and invoking Poincar\'e's
Lemma on the contractible domain $\omega$.

For every $s,t=1\ldots d$ we apply the above criterion to $B=[F_{ij,st}]_{i,j=1\ldots
  d}$. Since the first and third conditions in (\ref{sym_B}) validate
the skew-symmetry of $B$ and (\ref{tre}), we get existence of $d^2$ vector fields
$\phi_{st}=[\phi^j_{st}]_{j=1\ldots d}$  on $\omega$, satisfying:
\begin{equation}\label{quattro}
F_{ij,st} = \partial_i\phi_{st}^j - \partial_j\phi^i_{st} \quad\mbox{ for all }\; i,j,s,t=1\ldots d.
\end{equation}
By the first condition in (\ref{sym_B}), we note that
$\partial_i(\phi^j_{st} + \phi^j_{ts}) - \partial_j(\phi^i_{st} +
\phi^i_{ts}) = 0$ for all $i,j,s,t$, which implies that
each $\phi_{st}+\phi_{ts}$ is a gradient. Thus, without loss of generality 
we may take:
$$\phi_{st} = -\phi_{ts} \quad\mbox{ for all }\; s,t=1\ldots d.$$

For every $t=1\ldots d$ consider now the skew-symmetric matrix field
$B=[\phi^j_{st} - \phi^s_{jt}]_{j,s=1\ldots d}$. Condition (\ref{tre})
holds, in virtue of (\ref{quattro}) and the second condition in (\ref{sym_B}):
$$\partial_i(\phi^j_{st} - \phi^s_{jt}) + \partial_j(\phi^s_{it} - \phi^i_{st}) + \partial_s(\phi^i_{jt} - \phi^j_{it})
= F_{ij,st} + F_{si,jt} + F_{js,it}=0,$$
and so there follows existence of vector fields $\eta_t = [\eta_t^s]_{s=1\ldots d}$ on $\omega$, such that:
\begin{equation}\label{cinque}
\phi^j_{st} - \phi^s_{jt} = \partial_{j}\eta^s_t-\partial_s\eta^t_j \quad\mbox{ for all }\; j,s,t=1\ldots d.
\end{equation}
We now finally define:
$$A_{ij} = -\frac{1}{2}(\eta^i_j + \eta^j_i) \quad\mbox{ for all }\; i,j=1\ldots d.$$
The matrix field $A=[A_{ij}]_{i,j=1\ldots d}$ is obviously symmetric,
and from (\ref{cinque}) and (\ref{quattro}) we get:
\begin{equation*}
\begin{split}
2\mathfrak{C}^2(A)_{ij,st} & = -\partial_i\partial_s(\eta^j_t+\eta^t_j) - \partial_j\partial_t(\eta^i_s+\eta^s_i)
+ \partial_i\partial_t(\eta^j_s+\eta^s_j) + \partial_j\partial_s(\eta^i_t+\eta^t_i)
\\ & = \partial_t(\phi^i_{js} - \phi^j_{is}) + \partial_j(\phi^s_{ti} - \phi^t_{si}) 
+ \partial_s(\phi^j_{it} - \phi^i_{jt}) + \partial_i(\phi^t_{sj} - \phi^s_{tj})  
\\ & = F_{ti,js} + F_{js,ti} + F_{tj,si} + F_{si,tj} = 2(F_{ti,js} +F_{tj,si}) = 
2 F_{st,ij} = 2F_{ij,st}
\end{split}
\end{equation*}
for all $i,j,s,t=1\ldots d$, where in the last three equalities above we used the first and second
conditions in (\ref{sym_B}). The proof is done.
\endproof

\medskip

\noindent Observe that for $d=3$ and $k=1$, any choice of $6$ functions $F_{12,12},
F_{12,13}, F_{12,23}, F_{13,13}$, $F_{13,23}$, $F_{23,23}\in L^2(\omega,\R)$ gives
raise to $F\in L^2(\omega,\R^{81})$ satisfying (\ref{sym_B}). Indeed,
the first condition holds by defining the remaining components of $F$
appropriately, while the second and the third conditions are implied
automatically by these symmetries. In this case, (\ref{MA}) consists of 6 equations in a
single unknown $v\in\R$, while (\ref{VK}) consists of $6$ equations
in $4$ unknowns $(v,w)\in \R^4$. Although both formulations
seem to be largely overdetermined, this paper actually shows that the set of
their solutions is dense in the space of continuous functions on $\bar\omega$.

\medskip

\noindent We are now ready to give:

\smallskip

\noindent {\bf Proof of Theorem \ref{th_CI_weakMA}}

By the construction in Theorem \ref{th_solve}, there exists a
matrix field $A\in \mathcal{C}^{1,1}(\bar\omega,\R^{d\times d}_\sym)$ such that
$\mathfrak{C}^2(A) = -F$. Given $v\in \mathcal{\mathcal{C}}^1(\bar\omega,\R^k)$, we 
apply Theorem \ref{th_final}, or Theorem \ref{th_finalK} in case of
codimension $k\geq 2d_*$, to $\epsilon=1/n$ and $v$, $w=0$,
$A+C\Id_d$. The constant $C>0$ is taken large enough to have, in the sense of matrix inequalities:
$$A + {C}\Id_d > \frac{1}{2}(\nabla v)^T\nabla v \quad \mbox{ on }\; \bar\omega.$$
The resulting $v_n=\tilde v$ provides the $n$-th
member of the claimed approximating sequence for $v$. When $v\in\mathcal{C}^0(\bar\omega,\R^k)$,
the sequence is obtained using a density argument.
\endproof

\section{Energy scaling bound for thin multidimensional films: proof
  of Theorem \ref{th_scaling}}\label{sec_appli}

In this section, we estimate the infimum of the energy
$\mathcal{E}^h(u)$ defined in (\ref{Eh}), interpreted as the averaged pointwise deficit of a weakly
regular immersion $u$ from being the
orientation preserving isometric immersion of the metric $g^h$ on
$\Omega^h$. When $d=2$, $k=1$ then $\mathcal{E}^h(u)$ is
the elastic energy (per unit thickness) of the deformation $u$ of a thin
film with midplate $\omega$, thickness $2h$, elastic energy density
$W$, and the prestrain tensor $g^h$. 

\medskip

\noindent {\bf Proof of Theorem \ref{th_scaling}}

{\bf 1.} Fix $\alpha\in \big(0,\frac{1}{1+s}\big)$. By Theorem \ref{th_final}, 
there exists $v\in\mathcal{C}^{1,\alpha }(\bar\omega,\R^k)$
and $w\in \mathcal{C}^{1,\alpha}(\bar\omega,\R^d)$, solving (\ref{VK})
with the right hand side given by the $d\times d$ principal minor of $S$:
\begin{equation}\label{exactMAM}
\frac{1}{2}(\nabla v)^T\nabla v + \sym\nabla w = S_{d\times d}.
\end{equation}
We regularize $v,w$ to
$v_\epsilon\in\mathcal{C}^\infty(\bar\omega,\R^k)$, $w_\epsilon\in \mathcal{C}^\infty(\bar\omega,\R^d)$
convolving with the kernels $\{\phi_\epsilon(x) \}_{\epsilon\to 0}$ as in Lemma \ref{lem_stima},
where  $\epsilon$ is a positive power $t$ of $h$, to be chosen later:
$$ v_\epsilon= v *\phi_\epsilon,\quad w_\epsilon=w*\phi_\epsilon, \quad \epsilon=h^t.$$
By \ref{stima2} and a version of \ref{stima4} in: $\|(fg)*\phi_\epsilon - (f*\phi_\epsilon)
(g*\phi_\epsilon)\|_0\leq C\epsilon^{2\alpha} \|f\|_{0,\alpha} \|g\|_{0,\alpha} $, we get:
\begin{equation*}
\begin{split}
\big\| \frac{1}{2}(\nabla v_\epsilon)^T&\nabla v_\epsilon  +
\sym\nabla w_\epsilon - S_{d\times d}\big\|_0\\ & \leq \big\|
\frac{1}{2}(\nabla v_\epsilon)^T\nabla v_\epsilon + \mathrm{sym}\,\nabla w_\epsilon - S_{d\times
  d}*\phi_\epsilon\big\|_0 + \| S_{d\times d}*\phi_\epsilon- S_{d\times d}\|_0.
\end{split}
\end{equation*}
Since $\mathrm{sym}\,\nabla w_\epsilon - S_{d\times d}*\phi_\epsilon=
- \frac{1}{2}\big((\nabla v)^T\nabla v\big)*\phi_\epsilon$, this leads to:
\begin{equation}\label{m1}
\begin{split}
\big\| \frac{1}{2}(\nabla v_\epsilon)^T&\nabla v_\epsilon  +
\sym\nabla w_\epsilon - S_{d\times d}\big\|_0\leq
C\epsilon^{2\alpha}\|\nabla v\|_{0,\alpha}^2 + C\epsilon^2 \|\nabla^2S_{d\times d}\|_0\leq C\epsilon^{2\alpha}.
\end{split}
\end{equation}
 Further, by \ref{stima1} and a version of \ref{stima2} in: $\|\nabla
 (f-f*\phi_\epsilon)\|_0\leq C\epsilon^{\alpha-1}\|f\|_{0,\alpha}$, we obtain:
\begin{equation}\label{m2}
\|\nabla v_\epsilon\|_0 + \|\nabla w_\epsilon\|_0 \leq C,\qquad 
\|\nabla^2v_\epsilon\|_0+ \|\nabla^2w_\epsilon\|_0 \leq C\epsilon^{\alpha-1}.
\end{equation}

\medskip

{\bf 2.} Denote $\delta=\gamma/2$ and define $u^h\in \mathcal{C}^\infty(\bar \Omega^h,\R^{d+k})$ as follows: 
\begin{equation*}\label{recseq0}
\arraycolsep=3.4pt\def\arraystretch{1.4}
\begin{split}
u^h(x,z) = \; &  id_{d+k} + h^{\delta/2}\left[\begin{array}{c} 0\\
    v_\epsilon\end{array}\right] + h^{\delta}\left[\begin{array}{c} w_\epsilon\\ 0\end{array}\right] 
\\ & + \Big(h^{\delta/2} \left[\begin{array}{c} -(\nabla v_\epsilon)^T\\ 0\end{array}\right]
+ h^{\delta} \left[\begin{array}{c} 2 S_{d\times k}\\ S_{k\times
      k}-\frac{1}{2}(\nabla v_\epsilon)(\nabla v_\epsilon)^T\end{array}\right] +h^{3\delta/2} B(x)\Big)z
\\ & \hspace{-1.8cm} \mbox{where we denote: }\; 
S=\left[\begin{array}{c|c} S_{d\times  d} & S_{d\times k} 
    \\ \hline S_{k\times d} & S_{k\times k}\end{array}\right],
\end{split}
\end{equation*}
and where the higher order correction field $B\in\mathcal{C}^\infty (\bar\omega, \R^{(d+k)\times k})$ is given by:
$$\arraycolsep=3.4pt\def\arraystretch{1.4} 
B(x) = \left[\begin{array}{c} -(\nabla v_\epsilon)^TS_{k\times k}+\frac{1}{2}(\nabla
v_\epsilon)^T(\nabla v_\epsilon)(\nabla v_\epsilon)^T + (\nabla w_\epsilon)^T(\nabla
v_\epsilon)^T\\  \hline 2\sym \big((\nabla v_\epsilon) S_{d\times k}\big)
\end{array}\right].$$

It follows that for all $x\in\bar\omega$ and $z\in B(0,1)\subset\R^k$ there holds:
\begin{equation*}
\arraycolsep=3.4pt\def\arraystretch{1.4}
\begin{split}
\nabla u^h(x, hz) = & \;\Id_{d+k} + h^{\delta/2} \left[\begin{array}{c|c} 
   0& -(\nabla v_\epsilon)^T\\ \hline \nabla v_\epsilon &0 \end{array}\right] +
h^\delta \left[\begin{array}{c|c} \nabla w_\epsilon & 2S_{d\times k} \\ \hline 0 &
    S_{k\times k}-\frac{1}{2}(\nabla v_\epsilon)(\nabla  v_\epsilon)^T\end{array}\right] 
\\ & + h^{3\delta/2} \left[\begin{array}{c|c}   0& B \end{array}\right]
- h^{1+\delta/2} \left[\begin{array}{c|c} \big[\langle\partial_i\partial_j
    v_\epsilon,z\rangle\big]_{i,j=1\ldots d} &0\\ \hline 0&0\end{array}\right]
\\ & +{O}\big(h^{1+\delta})(1+ \|\nabla^2v_\epsilon\|_0)+O(h^{1+3\delta/2})\|\nabla^2w_\epsilon\|_0.
\end{split}
\end{equation*}
We now observe that: $(g^h)^{-1/2}=\Id_{d+k} - h^\delta S + {O}(h^{2\delta})$, and proceed with computing:
\begin{equation*}
\arraycolsep=3.4pt\def\arraystretch{1.4}
\begin{split}
\big(\nabla u^h(g^h)^{-1/2}\big)(x, hz) = & \;\Id_{d+k} + P^h +
h^\delta \left[\begin{array}{c|c} \nabla w_\epsilon -S_{d\times d} & 0 \\ \hline 0 &
-\frac{1}{2}(\nabla v_\epsilon)(\nabla v_\epsilon)^T\end{array}\right] \\ & +
h^{3\delta/2} \left[\begin{array}{c|c} (\nabla v_\epsilon)^T S_{k\times d} &
    \frac{1}{2}(\nabla v_\epsilon)(\nabla v_\epsilon)^T\nabla v_\epsilon + (\nabla
    w_\epsilon)^T(\nabla v_\epsilon)^T\\ \hline -(\nabla v_\epsilon)S_{d\times d} &
S_{k\times d} (\nabla v_\epsilon)^T\end{array}\right] \\ &
-  h^{1+\delta/2}\left[\begin{array}{c|c} \big[\langle\partial_i\partial_j
    v_\epsilon,z\rangle\big]_{i,j=1\ldots d}  &0\\ \hline 0&0\end{array}\right]\\ & 
+{O}(h^{2\delta})  +{O}(h^{1+\delta}) (1+ \|\nabla^2v_\epsilon\|_0+\|\nabla^2w_\epsilon\|_0).
\end{split}
\end{equation*}
Above, we used the following skew-symmetric matrix field:
$$\arraycolsep=3.4pt\def\arraystretch{1.4} 
P^h = \left[\begin{array}{c|c} 0 & p^h \\ \hline -(p^h)^T & 0\end{array}\right],
\qquad p^h= - h^{\delta/2} (\nabla v_\epsilon)^T+h^\delta S_{d\times k}. $$
For future purpose, it is convenient to compute:
$$ \arraycolsep=3.4pt\def\arraystretch{1.4}
(P^h)^2 = - h^\delta\left[\begin{array}{c|c} (\nabla v_\epsilon)^T\nabla v_\epsilon
& 0 \\ \hline 0 & (\nabla v_\epsilon)(\nabla v_\epsilon)^T \end{array}\right] 
+ 2h^{3\delta/2}\sym \left[\begin{array}{c|c} (\nabla v_\epsilon)^TS_{k\times d} & 0 \\ \hline
0 & (\nabla v_\epsilon)S_{d\times k}\end{array}\right] 
+ O(h^{2\delta}).$$

\medskip

{\bf 3.}  Consider the rotation fields $Q^h\in\mathcal{C}^\infty(\bar \omega, \mathrm{SO}(d+k))$, defined by:
$$Q^h = \exp(-P^h) = \Id_{d+k} - P^h +\frac{1}{2}(P^h)^2-\frac{1}{6}(P^h)^3 + {O}(h^{2 \delta}). $$
Then we get:
\begin{equation*}
\arraycolsep=3.4pt\def\arraystretch{1.4}
\begin{split}
\big(Q^h \nabla u^h(g^h)^{-1/2}\big)(x, hz) = \; & \Id_{d+k} + h^\delta \left[\begin{array}{c|c} \frac{1}{2}(\nabla
    v_\epsilon)^T\nabla v_\epsilon+\nabla w_\epsilon -S_{d\times d} &  0 \\ \hline 0 &
0 \end{array}\right] \\ & + h^{3\delta/2} \left[\begin{array}{c|c}
\mbox{skew}\big((\nabla v_\epsilon)^TS_{k\times d}\big) &  (\nabla
w_\epsilon)^T(\nabla v_\epsilon)^T\\ \hline -  (\nabla v_\epsilon)\nabla
w_\epsilon & \mbox{skew}\big(S_{k\times d}(\nabla v_\epsilon)^T\big) \end{array}\right] +\frac{1}{3}(P^h)^{3} \\ & 
- h^{1+\delta/2}\left[\begin{array}{c|c} \big[\langle\partial_i\partial_j
    v_\epsilon, z\rangle\big]_{i,j=1\ldots d} &0\\ \hline 0&0\end{array}\right] \\ & 
+ {O}(h^{2\delta})  + {O}(h^{1+\delta} )(1+\|\nabla^2v_\epsilon\|_0+\|\nabla^2w_\epsilon\|_0).
\end{split}
\end{equation*}
Finally, we apply another rotation field $\bar Q^h\in \mathcal{C}^\infty(\bar \omega, \mathrm{SO}(d+k))$:
\begin{equation*}
\arraycolsep=3.4pt\def\arraystretch{1.4}
\begin{split}
& \bar Q^h = \exp(-\bar P^h) = \Id_{d+k} - \bar P^h + {O}(h^{2 \delta}), \\
& \mbox{where }\; \bar P^h = \left[\begin{array}{c|c}
    \mbox{skew}\big(h^\delta\nabla w_\epsilon  + h^{3\delta/2}(\nabla  v_\epsilon)^TS_{k\times d}\big)
&  h^{3\delta/2} (\nabla w_\epsilon)^T(\nabla v_\epsilon)^T \\ \hline - h^{3\delta/2} (\nabla v_\epsilon)\nabla w_\epsilon &
h^{3\delta/2} \mbox{skew}\big(S_{k\times d}(\nabla v_\epsilon)^T\big)\end{array}\right] + \frac{1}{3}(P^h)^3,
\end{split}
\end{equation*}
to get:
\begin{equation*}
\arraycolsep=3.4pt\def\arraystretch{1.4}
\begin{split}
\big(\bar Q^h Q^h \nabla u^h(g^h)^{-1/2}\big)&(x, hz) = \Id_{d+k} +
h^\delta \left[\begin{array}{c|c} \frac{1}{2}(\nabla
 v_\epsilon)^T\nabla v_\epsilon+\sym\nabla w_\epsilon -S_{d\times d} & 0 \\ \hline 0 & 0 \end{array}\right] 
\\ & - h^{1+\delta/2}\left[\begin{array}{c|c} \big[\langle\partial_i\partial_j
    v_\epsilon,z\rangle\big]_{i,j=1\ldots d}  &0\\ \hline 0&0\end{array}\right]\\ & 
+ {O}(h^{2\delta}) + {O}(h^{1+\delta}) (1+\|\nabla^2v_\epsilon\|_0+\|\nabla^2w_\epsilon\|_0).
\end{split}
\end{equation*}

\medskip

{\bf 4.} In conclusion, we obtain the following energy bound below,
valid provided that we may use Taylor's expansion of $W$ up to second
order in perturbation of $\Id_{d+k}$, which here holds when
$h^{1+\delta/2}(\|\nabla^2v_\epsilon\|_0+\|\nabla^2w_\epsilon\|_0\big)\to
0$ as $h\to 0$, implied by $\lim_{h\to 0} \big(h^{1+\delta/2}\epsilon^{\alpha-1}\big)=0$:
\begin{equation*}
\begin{split}
& \inf \mathcal{E}^h \leq \mathcal{E}^h(u^h) = \fint_{\Omega^1}W\Big(\bar Q^h Q^h \nabla
u^h(g^h)^{-1/2}(x, hz) \Big)\;\mbox{d}(x,z) \\ & ~~ \leq
C\fint_{\Omega^1} \Big(h^{2\delta} \big|\frac{1}{2}(\nabla
 v_\epsilon)^T\nabla v_\epsilon+\sym\nabla w_\epsilon -S_{d\times d}
 \big|^2 + h^{2+\delta}(\|\nabla^2v_\epsilon\|_0^2 
+\|\nabla^2w_\epsilon\|_0^2) + h^{4\delta}\Big)\;\mbox{d}(x, z).
\end{split}
\end{equation*}
Recalling (\ref{m1}) and (\ref{m2}), the obtained bound further leads to:
$$\inf \mathcal{E}^h \leq C\big(h^{2\delta}\epsilon^{4\alpha}+h^{2+\delta}\epsilon^{2\alpha-2}+h^{4\delta}\big)
= C\big(h^{2\delta+4\alpha t}+h^{2+\delta+(2\alpha-2)t}+h^{4\delta}\big).$$

Minimizing the right hand side above is equivalent to
maximizing the minimal of the three exponents. For $\delta<2$, we hence
choose the exponent $t$ in $\epsilon=h^t$ so that
${2\delta+4\alpha t}={2+\delta+(2\alpha-2)t}$, namely
$t=\frac{2-\delta}{2\alpha+2}$. Consequently, we get:
$$\inf \mathcal{E}^h \leq C\big(h^{2\frac{\delta+2\alpha}{\alpha+1}}+ h^{4\delta}\big)\leq
C\big(h^{\big(4+ 2\frac{(1+s)(\delta-2)}{2+s}\big)-} +h^{4\delta}\big) $$
upon recalling the range of admissible exponents $\alpha$.  On the
other hand, when $\delta\geq 2$, then we choose $t$ close to $0$.
The conclusion of Theorem \ref{th_scaling} follows by a direct inspection.
\endproof

\end{document}